\documentclass[journal]{IEEEtran}

\usepackage{extarrows,paralist}
\usepackage{mhchem} 
\usepackage{siunitx} 
\usepackage{graphicx} 
\usepackage{amsmath,graphicx}
\usepackage{mathtools}

\usepackage{amsthm}
\newtheorem{lem}{Lemma}
\newtheorem{theorem}{Theorem}

\newtheorem{assump}{Assumption}

\usepackage[square, comma, sort&compress, numbers]{natbib}
\usepackage{mathtools}
\usepackage{graphicx}
\usepackage{float}
\usepackage{caption}
\usepackage{color}
\usepackage{placeins}
\usepackage{float}
\usepackage{tabularx,colortbl}
\usepackage[skip=2pt,font=footnotesize]{caption}
\usepackage{amsmath,subfigure}
\usepackage{listings}
\usepackage[outdir=./]{epstopdf}
\usepackage{amssymb}
\usepackage{multirow}
\usepackage{algorithm}
\usepackage{algpseudocode}
\usepackage{hyperref}

\usepackage{enumitem}
\allowdisplaybreaks
\usepackage[dvipsnames]{xcolor}

\makeatletter
\newcommand{\vast}{\bBigg@{4}}
\newcommand{\Vast}{\bBigg@{5}}
\makeatother

\def\mbb{\mathbb}
\def\mb{\mathbf}
\def\mc{\mathcal}

\def\ul{\underline}

\def\bs{\boldsymbol}

\begin{document}
\title{\bf\LARGE ADD-OPT: Accelerated Distributed Directed Optimization} 
\author{Chenguang Xi,~\emph{Student Member,~IEEE}, Ran Xin,~\emph{Student Member,~IEEE}, \\and Usman A. Khan,~\emph{Senior Member,~IEEE}
\thanks{
The authors are with the ECE Department at Tufts University, Medford, MA; {\texttt{chenguang.xi@tufts.edu, khan@ece.tufts.edu}}. This work has been partially supported by an NSF Career Award \# CCF-1350264.}
}
	
\maketitle
\begin{abstract}
In this paper, we consider distributed optimization problems where the goal is to minimize a sum of objective functions over a multi-agent network. We focus on the case when the inter-agent communication is described by a strongly-connected, \emph{directed} graph. The proposed algorithm, ADD-OPT (Accelerated Distributed Directed Optimization), achieves the best known convergence rate for this class of problems,~$O(\mu^{k}),0<\mu<1$, given strongly-convex, objective functions with globally Lipschitz-continuous gradients, where~$k$ is the number of iterations. Moreover, ADD-OPT supports a wider and more realistic range of step-sizes in contrast to existing work. In particular, we show that ADD-OPT converges for arbitrarily small (positive) step-sizes. Simulations further illustrate our results.
\end{abstract}
	

\vspace{-0.2cm}
\section{Introduction}\label{s1}
In this paper, we consider distributed optimization problems where the goal is to minimize a sum of objective functions over a multi-agent network. Formally, we consider a decision variable,~$\mb{z}\in\mbb{R}^p$, and a strongly-connected network containing~$n$ agents, where each agent,~$i$, only has access to a local objective function,~$f_i:\mbb{R}^p\rightarrow\mbb{R}$. The goal is to have each  agent minimize the sum of objectives,~$\sum_{i=1}^nf_i(\mb{z})$, via information exchange with the neighbors. This formulation has gained great interest due to its widespread applications in, e.g., large-scale machine learning,~\cite{distributed_Boyd,ml}, model-predictive control,~\cite{distributed_Necoara}, cognitive networks,~\cite{distributed_Mateos,distributed_Bazerque}, source localization,~\cite{distributed_Rabbit,distributed_Khan}, resource scheduling,~\cite{distributed_Chunlin}, and message routing,~\cite{distributed_Neglia}.

Most of the existing algorithms assume information exchange over undirected networks (graphs), where the communication between the agents is bidirectional, i.e., if agent~$i$ sends information to agent~$j$ then agent~$j$ can also send information to agent~$i$. Related work includes Distributed Gradient Descent (DGD),~\cite{uc_Nedic,cc_nedic,cc_Lobel2,cc_Ram2}, which achieves~$O(\frac{\ln k}{\sqrt{k}})$ convergence for arbitrary convex functions, and~$O(\frac{\ln k}{k})$ for strongly-convex functions, where~$k$ is the number of iterations. The convergence rates can be accelerated with an additional Lipschitz-continuity assumption on  the associated gradient. For example, see DGD~\cite{DGD_Yuan} that converges at~$O(\frac{1}{k})$ for general convex functions but within a ball around the optimal solution, whereas, it converges linearly to the optimal solution for strongly-convex functions. The distributed Nestrov's method,~\cite{fast_Gradient}, converges at~$O(\frac{\ln k}{k^2})$ for general convex functions. Of significant relevance is EXTRA,~\cite{EXTRA}, which converges to the optimal solution at~$O(\frac{1}{k})$ for general convex functions and is linear for strongly-convex functions. The work in \cite{Augmented_EXTRA} improves EXTRA by relaxing the weight matrices to be asymmetric. Besides the gradient-based methods, the distributed implementation of ADMM,~\cite{ADMM_Mota,ADMM_Wei, ADMM_Shi}, has also been considered over undirected graphs.

The aforementioned methods,~\cite{uc_Nedic,cc_nedic,cc_Lobel2,cc_Ram2,DGD_Yuan,fast_Gradient,EXTRA,Augmented_EXTRA,ADMM_Mota,ADMM_Wei, ADMM_Shi}, are applicable to undirected graphs that allow the use of doubly-stochastic weight matrices; row-stochasticity guarantees that all agents reach consensus, while the column-stochasticity ensures that each local gradient contributes equally to the global objective,~\cite{D-DGD}. On the contrary, when the underlying graph is directed, the weight matrix may only be row-stochastic or column-stochastic but not both. In this paper, we provide a distributed optimization algorithm that does not require doubly-stochastic weights and thus is applicable to directed graphs (digraphs). See~\cite{GHARESIFARD2012539,7100912} for work on balancing the weights in strongly-connected digraphs.

Optimization in continuous-time over weight-balanced digraphs has been studied earlier in~\cite{weightbalanceddigraphs1,weightbalanceddigraphs2}. Existing discrete-time algorithms include the following: Gradient-Push (GP), \cite{opdirect_Nedic,opdirect_Tsianous,opdirect_Tsianous2,opdirect_Tsianous3}, that combines DGD, \cite{uc_Nedic}, and push-sum consensus,~\cite{ac_directed0,ac_directed}; Directed-Distributed Gradient Descent (D-DGD),~\cite{D-DGD,D-DPS}, which uses Cai and Ishii's work on surplus consensus,~\cite{ac_Cai1}, and combines it with DGD; and~\cite{opdirect_Makhdoumi}, where the authors apply the weight-balancing technique,~\cite{c_Hooi-Tong}, to DGD. These gradient-based methods,~\cite{opdirect_Nedic,opdirect_Tsianous,opdirect_Tsianous2,opdirect_Tsianous3,D-DGD,D-DPS,opdirect_Makhdoumi}, restricted by the diminishing step-size, converge relatively slowly at~$O(\frac{\ln k}{\sqrt{k}})$. When the objective functions are strongly-convex, the convergence rate can be accelerated to~$O(\frac{\ln k}{k})$,~\cite{opdirect_Nedic3}.


A recent paper proposed a fast distributed algorithm, termed DEXTRA, \cite{DEXTRA,the_copy_work}, to solve the distributed consensus optimization problem over directed graphs. By combining the push-sum protocol, \cite{ac_directed0,ac_directed}, and EXTRA,~\cite{EXTRA}, DEXTRA achieves a linear convergence rate given that the objective functions are strongly-convex. However, a limitation of DEXTRA is a restrictive step-size range, i.e., the greatest lower bound of DEXTRA's step-size is strictly greater than zero. In particular, DEXTRA requires the step-size,~$\alpha$, to follow~$\alpha\in(\underline{\alpha},\overline{\alpha})$, where~$\underline{\alpha}>0$. Estimating~$\underline{\alpha}$ in a distributed setting is challenging because it may require global knowledge. In contrast if~$\underline{\alpha}=0$, agents can pick a small enough positive constant to ensure the convergence. In this paper, we propose ADD-OPT (Accelerated Distributed Directed Optimization) to address the step-size limitation inherent to DEXTRA. In particular, ADD-OPT's step-size follows~$\alpha\in(0,\overline{\alpha})$, i.e.,~$\underline{\alpha}=0$, ensuring that the lower bound of ADD-OPT's step-size does not require any global knowledge. We show that ADD-OPT converges linearly for strongly-convex functions.

The remainder of the paper is organized as follows. Section~\ref{s2} formulates the problem and describes ADD-OPT. We also present appropriate assumptions in Section \ref{s2}. Section~\ref{s3} states the main convergence results. In Section~\ref{s4}, we present some lemmas as the basis of the proof of ADD-OPT's convergence. The proof of main results is provided in Section~\ref{s5}. We show numerical results in Section~\ref{s6} and Section~\ref{s7} contains the concluding remarks.

\textbf{Basic Notation:} We use lowercase bold letters to denote vectors and uppercase italic letters to denote matrices. The matrix,~$I_n$, represents the~$n\times n$ identity;~$\mb{1}_n$ and~$\mb{0}_n$ are the~$n$-dimensional column vectors of all~$1$'s and~$0$'s, respectively. We denote by~$A\otimes B$, the Kronecker product of two matrices,~$A$ and~$B$. For any~$f(\mb{x})$,~$\nabla f(\mb{x})$ denotes the gradient of~$f$ at~$\mb{x}$. The spectral radius of a matrix,~$A$, is represented by~$\rho(A)$. \begin{color}{black}For an irreducible, column-stochastic matrix,~$A$, we denote its right and left eigenvectors corresponding to the eigenvalue of~$1$ by $\bs{\pi}$ and~$\mb{1}_n^\top$, respectively, such that~$\mb{1}_n^\top\bs{\pi} = 1$. Depending on its argument, we denote~$\|\cdot\|$ as either a particular matrix norm, the choice of which will be clear in Lemma~\ref{lem3}, or a vector norm that is compatible with this particular matrix norm, i.e.,~$\|A\mb{x}\|\leq\|A\|\|\mb{x}\|$ for all matrices,~$A$, and all vectors,~$\mb{x}$. The notation~$\|\cdot\|_2$ denotes the Euclidean norm of vectors and matrices. Since all vector norms on finite-dimensional vector space are equivalent, we have the following:~$c{'}\|\cdot\|\leq\|\cdot\|_2\leq c\|\cdot\|, d{'}\|\cdot\|_2\leq\|\cdot\|\leq d\|\cdot\|_2$, where~$c{'},c,d{'},d$ are some positive constants. \end{color}

\section{ADD-OPT Development}\label{s2}
In this section, we formulate the optimization problem and describe ADD-OPT. We first derive an informal but intuitive proof showing that ADD-OPT enables the agents to achieve consensus and reach the optimal solution of Problem P1, described below. After propose ADD-OPT, we relate it to DEXTRA and discuss the applicable range of step-sizes. Formal convergence results are deferred to Sections~\ref{s3}.

Consider a strongly-connected network of~$n$ agents communicating over a directed graph,~$\mc{G}=(\mc{V},\mc{E})$, where~$\mc{V}$ is the set of agents, and~$\mc{E}$ is the collection of ordered pairs,~$(i,j),i,j\in\mc{V}$, such that agent~$j$ can send information to agent~$i$, $j\rightarrow i$. Define~$\mc{N}_i^{{\scriptsize \mbox{in}}}$ to be the collection of in-neighbors, i.e., the set of agents that can send information to agent~$i$. Similarly,~$\mc{N}_i^{{\scriptsize \mbox{out}}}$ is the set of out-neighbors of agent~$i$. Note that both~$\mc{N}_i^{{\scriptsize \mbox{in}}}$ and~$\mc{N}_i^{{\scriptsize \mbox{out}}}$ include node~$i$. Note that in a directed graph when~$(i,j)\in\mc{E}$, it is not necessary that~$(j,i)\in\mc{E}$. Consequently,~$\mc{N}_i^{{\scriptsize \mbox{in}}}\neq\mc{N}_i^{{\scriptsize \mbox{out}}}$, in general. We assume that each agent~$i$ knows\footnote{Such an assumption is standard in the related literature, see e.g.,~\cite{opdirect_Nedic,opdirect_Tsianous,opdirect_Tsianous2,opdirect_Tsianous3,D-DGD,D-DPS,opdirect_Makhdoumi,DEXTRA}.} its out-degree (the number of out-neighbors), denoted by~$|\mc{N}_i^{{\scriptsize \mbox{out}}}|$; see~\cite{bullo_book} for details. We focus on solving a convex optimization problem that is distributed over the above multi-agent network. In particular, the network of agents cooperatively solves the following optimization problem:
\begin{align}
\mbox{P1}:
\quad\mbox{min  }&f(\mb{z})=\sum_{i=1}^nf_i(\mb{z}),\nonumber
\end{align}
where each local objective function,~$f_i:\mbb{R}^p\rightarrow\mbb{R}$ is known only by agent~$i$. We assume that each local function,~$f_i(z)$, is strongly-convex and differentiable, whereas the optimal solution of Problem P1 exists and is finite. Our goal is to develop a distributed algorithm such that each agent converges to the global solution of Problem P1 via exchanging information with nearby agents over a directed graph. We formalize the set of assumptions as follows. These assumptions are standard in the literature for optimization of smooth convex functions, see e.g.,~\cite{EXTRA,DEXTRA,DGD_Yuan}.
\begin{assump}\label{asp2}
	The communication graph,~$\mc{G}$, is a strongly-connected digraph. Each agent in the network has the knowledge of its out-degree.
\end{assump}

\begin{assump}[Lipschitz-continuous gradients and strong-convexity]\label{asp1}
	Each local function,~$f_i$, is differentiable and strongly-convex, and the gradient is globally Lipschitz-continuous, i.e., for any~$i$ and~$\mb{z}_1, \mb{z}_2\in\mbb{R}^p$,
	{\color{black}
	\begin{enumerate}[label=(\alph*)]
	    \item there exists a positive constant~$l$ such that
	    $$\|\mb{\nabla} f_i(\mb{z}_1)-\mb{\nabla} f_i(\mb{z}_2)\|_2\leq l\|\mb{z}_1-\mb{z}_2\|_2;$$
		\item there exists a positive constant~$s$ such that,
		$$f_i(\mb{z}_1)-f_i(\mb{z}_2)\leq\mb{\nabla} f_i(\mb{z}_1)^\top(\mb{z}_1-\mb{z}_2)-\frac{s}{2}\|\mb{z}_1-\mb{z}_2\|_2^2.$$ 
	\end{enumerate}
    Clearly, the Lipschitz-continuity and strongly-convexity constants for the global objective function~$f(\mb{z})$ are~$nl$ and~$ns$, respectively.
    }
\end{assump}

\begin{assump}\label{asp3}
	The optimal solution exists and is bounded and unique. In particular, we denote~$\ul{\mb{z}}^*\in\mbb{R}^p$ the optimal solution, i.e., 
	$$\ul{\mb{z}}^*=\min_{\mb{z}\in\mbb{R}^p}f(\mb{z}).$$
\end{assump}

\subsection{ADD-OPT Algorithm}
To solve Problem P1, we describe the implementation of ADD-OPT as follows. Each agent,~$j\in\mc{V}$, maintains three vector variables:~$\mb{x}_{k}^j$,~$\mb{z}_{k}^j$,~$\mb{w}_{k}^j,$ all in~$\mbb{R}^p$, as well as a scalar variable,~$\mathsf{y}_{k}^j\in\mbb{R}$, where~$k$ is the discrete-time index. At the~$k$th iteration, agent~$j$ assigns a weight to its states:~$a_{ij}\mb{x}_{k}^j$,~$a_{ij}\mb{w}_{k}^j$, and~$a_{ij}\mathsf{y}_{k}^j$; and sends these to each of its out-neighbors,~$i\in\mc{N}_j^{{\scriptsize \mbox{out}}}$, where the weights,~$a_{ij}$'s are such that:
\begin{align}
a_{ij}&=\left\{
\begin{array}{rl}
>0,&i\in\mc{N}_j^{{\scriptsize \mbox{out}}},\\
0,&\mbox{otherwise},
\end{array}
\right.
\quad
\sum_{i=1}^na_{ij}=1,\forall j.\label{a}
\end{align}
With agent~$i$ receiving the information from its in-neighbors, it updates~$\mb{x}_{k+1}^i$,~$\mathsf{y}_{k+1}^i$,~$\mb{z}_{k+1}^i$ and~$\mb{w}_{k+1}^i$ as follows:
\begin{subequations}\label{alg1}
	\begin{align}
	\mb{x}_{k+1}^i=&\sum_{j\in\mc{N}_i^{{\tiny \mbox{in}}}}a_{ij}\mb{x}_{k}^j-\alpha\mb{w}_{k}^i,\label{alg1a}\\
	\mathsf{y}_{k+1}^i=&\sum_{j\in\mc{N}_i^{{\tiny \mbox{in}}}}a_{ij}\mathsf{y}_{k}^j,\label{alg1b}\\
	\mb{z}_{k+1}^i=&\frac{\mb{x}_{k+1}^i}{\mathsf{y}_{k+1}^i},\label{alg1c}\\
	\mb{w}_{k+1}^i=&\sum_{j\in\mc{N}_i^{{\tiny \mbox{in}}}}a_{ij}\mb{w}_{k}^j+\nabla f_i(\mb{z}_{k+1}^i)-\nabla f_i(\mb{z}_{k}^i).\label{alg1d}
	\end{align}
\end{subequations}
In the above,~$\nabla f_i(\mb{z}_{k}^i)$ is the gradient of~$f_i(\mb{z})$ at~$\mb{z}=\mb{z}_{k}^i$. The step-size,~$\alpha$, is a positive number within a certain interval. We will explicitly show the range of~$\alpha$ in Section~\ref{s3}. For any agent~$i$, it is initialized with arbitrary vectors,~$\mb{x}_{0}^i$ and $\mb{z}_0^i$,~$\mb{w}_{0}^i=\nabla f_i(\mb{z}_{0}^i)$, and~$\mathsf{y}_{0}^i=1$. It is worth noting that~$\mathsf{y}_{k}^i\neq0$,~$\forall k$, given its initial condition and Assumption~\ref{asp2}, \cite{eig2}. We note that Eq.~\eqref{a} leads to a column-stochastic weight matrix,~$\ul{A}=\{a_{ij}\}$, by only requiring each agent to know its out-degree. It is indeed possible to construct such weights, e.g., by choosing
\begin{align}
a_{ij}&=\left\{
\begin{array}{rl}
1/|\mc{N}_j^{{\scriptsize \mbox{out}}}|,&i\in\mc{N}_j^{{\scriptsize \mbox{out}}},\\
0,&\mbox{otherwise},
\end{array}
\right.
\quad
,\forall j.\label{aa}
\end{align}

For analysis purposes, we now write Eq.~\eqref{alg1} in a matrix form. We use the following notation:
\begin{eqnarray}\label{not1}
\mb{x}_k=
\left[
\begin{array}{c}
\mb{x}_{k}^1\\
\vdots\\
\mb{x}_{k}^n
\end{array}
\right],~\mb{w}_k=
\left[
\begin{array}{c}
\mb{w}_{k}^1\\
\vdots\\
\mb{w}_{k}^n
\end{array}
\right],~\mb{z}_k=
\left[
\begin{array}{c}
\mb{z}_{k}^1\\
\vdots\\
\mb{z}_{k}^n
\end{array}
\right],\\\label{not2}
\nabla\mb{f}_k=
\left[
\begin{array}{c}
\nabla \mb{f}_1(\mb{z}_{k}^1)\\
\vdots\\
\nabla \mb{f}_n(\mb{z}_{k}^n)
\end{array}
\right],~\mb{y}_k=
\left[
\begin{array}{c}
\mathsf{y}_{k}^1\\
\vdots\\
\mathsf{y}_{k}^n
\end{array}
\right].
\end{eqnarray}
Let~$\ul{A}\in\mbb{R}^{n\times n}$ be the weighted adjacency matrix, i.e., the collection of weights,~$a_{ij}$; define
\begin{eqnarray}\label{not3}
A&=&\ul{A}\otimes I_p,\label{A}\\\label{Yk}
Y_k&=&\mbox{diag}\left(\mb{y}_k\right)\otimes I_p.
\end{eqnarray}
where `$\otimes$' is the Kronecker product. Clearly, we have $A,Y_k\in\mbb{R}^{np\times np}$, and~$A$ is a column-stochastic matrix. Given that~$\mb{y}_0=\mb{1}_n$, the graph,~$\mc{G}$, is strongly-connected and the corresponding weight matrix,~$\ul{A}$, is non-negative, $Y_k$ is invertible for any~$k$,~\cite{eig2}. Then, we can write Eq.~\eqref{alg1} in the matrix form, equivalently, as follows:
\begin{subequations}\label{alg1_matrix}
	\begin{align}
	\mb{x}_{k+1}=&A\mb{x}_{k}-\alpha\mb{w}_{k},\label{alg1_ma}\\
	\mb{y}_{k+1}=&\ul{A}\mb{y}_{k},\label{alg1_mb}\\
	\mb{z}_{k+1}=&Y_{k+1}^{-1}\mb{x}_{k+1},\label{alg1_mc}\\
	\mb{w}_{k+1}=&A\mb{w}_{k}+\nabla \mb{f}_{k+1}-\nabla \mb{f}_{k},\label{alg1_md}
	\end{align}
\end{subequations}
where we have the initial condition~$\mb{w}_0=\nabla\mb{f}_0$,~$\mb{y}_0=\mb{1}_n$.
\subsection{Interpretation of ADD-OPT}
Based on Eq.~\eqref{alg1_matrix}, we now give an intuitive interpretation on the convergence of ADD-OPT to the optimal solution. By combining Eqs. \eqref{alg1_ma} and \eqref{alg1_md}, we obtain that
 \begin{align}
 \mb{x}_{k+1}&=A\mb{x}_{k}-\alpha\left[A\mb{w}_{k-1}+\nabla \mb{f}_{k}-\nabla \mb{f}_{k-1}\right],\nonumber\\
 &=A\mb{x}_{k}-\alpha A\left[\frac{A\mb{x}_{k-1}-\mb{x}_k}{\alpha}\right]-\alpha\left[\nabla \mb{f}_{k}-\nabla \mb{f}_{k-1}\right],\nonumber\\
 &=2A\mb{x}_{k}-A^2\mb{x}_{k-1}-\alpha\left[\nabla \mb{f}_{k}-\nabla \mb{f}_{k-1}\right].\label{alg2}
 \end{align}
 Assume that the sequences generated by Eq.~\eqref{alg1_matrix} converge to their limits (note that this is not necessarily true), denoted by~$\mb{x}_\infty$,~$\mb{y}_\infty$,~$\mb{w}_\infty$,~$\mb{z}_\infty$,~$\nabla\mb{f}_\infty$, respectively. It follows from Eq.~\eqref{alg2} that
 \begin{align}\label{int_2}
\mb{x}_{\infty}=2A\mb{x}_{\infty}-A^2\mb{x}_{\infty}-\alpha\left[\nabla \mb{f}_{\infty}-\nabla \mb{f}_{\infty}\right],
 \end{align}
 which implies that~$(I_{np}-A)^2\mb{x}_\infty=\mb{0}_{np}$ or~$[(I_n-\ul{A})^2\otimes I_p]\mb{x}_\infty=\mb{0}_{np}$. Considering that~$\mb{y}_\infty=\ul{A}\mb{y}_\infty$, we obtain that~$\mb{x}_\infty\in\mbox{span}\{\mb{y}_\infty\otimes\mb{u}_p\}$ for some arbitrary~$p$-dimensional vector,~$\mb{u}_p$. Therefore, it follows that
 \begin{align}\label{int_3}
 \mb{z}_\infty=Y_{\infty}^{-1}\mb{x}_{\infty}\in\mbox{span}\{\mb{1}_n\otimes\mb{u}_p\},
 \end{align}
 where~$\mb{u}_p$ is some arbitrary~$p$-dimensional vector. The consensus is reached.
 
 By summing up the updates in Eq.~\eqref{alg2} over~$k$ from~$0$ to~$\infty$, we obtain that
 \begin{align}
 \mb{x}_\infty=A\mb{x}_\infty+\sum_{r=1}^\infty(A-I_{np})\mb{x}_r-\sum_{r=0}^\infty(A^2-A)\mb{x}_r-\alpha\nabla\mb{f}_\infty.\nonumber
 \end{align}
 Noting that~$\mb{x}_\infty=A\mb{x}_\infty$, it follows 
 \begin{align}
\alpha\nabla\mb{f}_\infty=\sum_{r=1}^\infty(A-I_{np})\mb{x}_r-\sum_{r=0}^\infty(A^2-A)\mb{x}_r.\nonumber
 \end{align}
 Therefore, we obtain that
\begin{align}
&\alpha(\mb{1}_n\otimes I_p)^\top\nabla\mb{f}_\infty\nonumber\\
&=\left(\mb{1}_n^\top(\ul{A}-I_{n})\otimes I_p\right)\sum_{r=1}^\infty\mb{x}_r-\left(\mb{1}_n^\top(\ul{A}^2-\ul{A})\otimes I_p\right)\sum_{r=0}^\infty\mb{x}_r,\nonumber\\
&=\mb{0}_p,\nonumber
\end{align}
which is the optimality condition of Problem P1 considering that~$\mb{z}_\infty\in\mbox{span}\{\mb{1}_n\otimes\mb{u}_p\}$. To summarize, if we assume that the sequences updated in Eq.~\eqref{alg1_matrix} have limits,~$\mb{x}_\infty$,~$\mb{y}_\infty$,~$\mb{w}_\infty$,~$\mb{z}_\infty$,~$\nabla\mb{f}_\infty$, we arrive at a conclusion that~$\mb{z}_\infty$ achieves consensus and reaches the optimal solution of Problem P1. We next discuss the relations between ADD-OPT and DEXTRA.

\subsection{ADD-OPT and DEXTRA}
Recent papers provide a fast distributed algorithm, termed DEXTRA \cite{DEXTRA,the_copy_work}, to solve Problem P1 over directed graphs. It achieves a linear convergence rate given that the objective functions are strongly-convex. At the~$k$th iteration of DEXTRA, each agent~$i$ keeps and updates three states,~$x_{k,i}$,~$\mathsf{y}_{k,i}$, and~$z_{k,i}$. The iteration, in matrix form, is shown as follows.
\begin{subequations}\label{dextra}
	\begin{align}
	\mb{x}_{k+1}=&\left(I_{np}+A\right)\mb{x}_k-\widetilde{A}\mb{x}_{k-1}-\alpha\left[\nabla\mb{f}_k-\nabla\mb{f}_{k-1}\right],\label{dextra1a}\\
	\mb{y}_{k+1}=&\underline{A}\mb{y}_k,\label{dextra1b}\\
	\mb{z}_{k+1}=&Y_{k+1}^{-1}\mb{x}_{k+1},\label{dextra1c}
	\end{align}
\end{subequations}
where~$\widetilde{A}$ is a column-stochastic matrix satisfying that~$\widetilde{A}=\theta I_{np}+(1-\theta)A$ with any~$\theta\in(0,\frac{1}{2}]$, and all other notation is the same as from earlier in this paper. 

By comparing Eqs.~\eqref{alg2} and \eqref{dextra1a},~\eqref{alg1_mb} and \eqref{dextra1b}, and~\eqref{alg1_mc} and \eqref{dextra1c}, it follows that the only difference between ADD-OPT and DEXTRA lies in the weighting matrices used when updating~$\mb{x}_k$. From DEXTRA to ADD-OPT, we change~$(I_{np}+A)$ in \eqref{dextra1a} to~$2A$ in \eqref{alg2}, and~$\widetilde{A}$ to~$A^2$, respectively. Mathematically, if~$A=I_{np}$, (equivalently~$\ul{A}=I_n$), the two algorithms are the same. With this modification, we will show in Section \ref{s3} that ADD-OPT supports a wider range of step-sizes as compared to DEXTRA, i.e., the greatest lower bound,~$\underline{\alpha}$, of ADD-OPT's step-size is zero while that of DEXTRA's is strictly positive. This also reveals the reason why in DEXTRA constructing~$\ul{A}$ to be an extremely diagonally-dominant matrix is preferred, see Assumption A2(c) in \cite{DEXTRA}. The more similar~$\ul{A}$ is to~$I_n$, the closer~$\underline{\alpha}$ approaches zero. However, in DEXTRA,~$\underline{\alpha}$ can never reach zero since~$\ul{A}$ cannot be the identity,~$I_n$, which otherwise means there is no communication between agents. In Section \ref{s5}, we provide a totally different proof, that is further more compact and elegant when compared to DEXTRA's analysis, to show the linear convergence rate of ADD-OPT.


\section{Main Result}\label{s3}
In this section, we analyze ADD-OPT with the help of the following notation. From Eqs.~\eqref{not1}-\eqref{Yk}, we further define~$\overline{\mb{x}}_k$,~$\overline{\mb{w}}_k$,~$\mb{z}^*$,~$\mb{g}_k$,~$\mb{h}_k\in\mbb{R}^{np}$ as 
\begin{align}\label{eqxb}
\overline{\mb{x}}_k&=\frac{1}{n}(\mb{1}_n\otimes I_p)(\mb{1}_n^\top\otimes I_p)\mb{x}_k,\\\label{eqwb}
\overline{\mb{w}}_k&=\frac{1}{n}(\mb{1}_n\otimes I_p)(\mb{1}_n^\top\otimes I_p)\mb{w}_k,\\
\mb{z}^*&=\mb{1}_n\otimes\ul{\mb{z}}^*,\\
\mb{g}_k&=\frac{1}{n}(\mb{1}_n\otimes I_p)(\mb{1}_n^\top\otimes I_p)\nabla\mb{f}_k,\\
\mb{h}_k&=\frac{1}{n}(\mb{1}_n\otimes I_p)(\mb{1}_n^\top\otimes I_p)\nabla\mb{f}(\overline{\mb{x}}_k),
\end{align}
where
\begin{eqnarray*}
\nabla\mb{f}(\overline{\mb{x}}_k)=
\left[
\begin{array}{c}
\nabla f_1(\frac{1}{n}(\mb{1}_n^\top\otimes I_p)\mb{x}_k)\\
\vdots\\
\nabla f_n(\frac{1}{n}(\mb{1}_n^\top\otimes I_p)\mb{x}_k)
\end{array}
\right],
\end{eqnarray*}
stacks its components in a column. We denote constants,~$\tau$,~$\epsilon$, and~$\eta$ as
{\color{black}
\begin{align}
\tau&=\left\|A-I_{np}\right\|_2,\label{parameter1}\\
\epsilon&=\left\|I_{np}-A_\infty\right\|_2,\label{parameter2}\\
\eta&=\max\left(\left|1-n\alpha l\right|,\left|1-n\alpha s\right|\right)\label{parameter3},
\end{align}
}
where~$A$ is the column-stochastic weight matrix used in Eq. \eqref{alg1_matrix},~$A_\infty=\lim_{k\rightarrow\infty}A^k$ represents~$A$'s limit,~$\alpha$ is the step-size, and~$l$ and~$s$ are respectively Lipschitz and strong-convexity constants from Assumption~\ref{asp1}. Let~$Y_\infty$ be the limit of~$Y_k$ in Eq. \eqref{Yk},
\begin{align}\label{yinf_eq}
Y_\infty=\lim_{k\rightarrow\infty}Y_k,
\end{align}
and~$y$ and~$y_-$ be the {\color{black}supremum of~$\|Y_k\|_2$ and~$\|Y_k^{-1}\|_2$} over~$k$, respectively, i.e.,
{\color{black}
\begin{align}
y&=\sup_{k}\left\|Y_k\right\|_2,\label{parameter5}\\
y_-&=\sup_{k}\left\|Y_k^{-1}\right\|_2.\label{parameter6}
\end{align}
}Note that the existence of the limits, $A_{\infty}$ and $Y_{\infty}$, will be clear in the following lemmas. Moreover, we define two constants,~$\sigma$, and,~$\gamma_1$, through the following two lemmas, which are related to the convergence of~$A$ and~$Y_\infty$. 
\begin{lem}\label{lem2}
	(Nedic \textit{et al}.~\cite{opdirect_Nedic}) Let Assumption \ref{asp2} hold. Consider~$Y_k$ and its limit~$Y_\infty$ as defined before. There exist~$0<\gamma_1<1$ and~$0<T<\infty$ such that for all~$k$	
	\begin{align}\label{DkDinfty1}
	\left\|Y_k-Y_\infty\right\|_2\leq T\gamma_1^{k}.
	\end{align}
\end{lem}
{\color{black}
\begin{lem}\label{lem3}
	Let Assumption \ref{asp2} hold. Consider~$Y_\infty$ in Eq.~\eqref{yinf_eq} with~$A$ being the column-stochastic matrix used in Eq. \eqref{alg1_matrix}. For any~$\mb{a}\in\mbb{R}^{np}$, define~$\overline{\mb{a}}=\frac{1}{n}(\mb{1}_n\otimes I_p)(\mb{1}_n^\top\otimes I_p)\mb{a}$. Then, there exists~$0<\sigma<1$ such that for all~$k$
	\begin{align}\label{sigma_eq}
	\left\|A\mb{a}-Y_\infty\overline{\mb{a}}\right\|\leq\sigma\left\|\mb{a}-Y_\infty\overline{\mb{a}}\right\|.
	\end{align}
\end{lem}
\begin{proof}
	First note that~$A=\ul{A}\otimes I_p$. Since~$\ul{A}$ is irreducible, column-stochastic with positive diagonals, from Perron-Frobenius theorem we note that~$\rho(\ul{A})=1$, every eigenvalue of~$\ul{A}$ other than~$1$ is strictly less than~$\rho(\ul{A})$, and~$\bs{\pi}$ is a strictly positive (right) eigenvector corresponding to the eigenvalue of~$1$ such that~$\mb{1}_n^\top\bs{\pi} = 1$; thus~$\lim_{k\rightarrow\infty} \ul{A}^k = \bs{\pi}\mb{1}_n^\top$. Recalling Eq.~$\eqref{A}$, we have $A_{\infty}=\lim_{k\rightarrow\infty}{A^k}=\lim_{k\rightarrow\infty}{(\ul{A}\otimes I_p)^k}=(\lim_{k\rightarrow\infty}{\ul{A}^k})\otimes I_p=(\bs{\pi}\mb{1}_n^\top)\otimes I_p.$ It follows that:
	\begin{align}
	AA_{\infty} &= (\ul{A}\otimes I_p)\Big((\bs{\pi}\mb{1}_n^\top)\otimes I_p\Big) = (\ul{A}\bs{\pi}\mb{1}_n^\top)\otimes I_p = A_{\infty}; \nonumber \\
    A_{\infty}A_{\infty} &= \Big((\bs{\pi}\mb{1}_n^\top)\otimes I_p\Big)\Big((\bs{\pi}\mb{1}_n^\top)\otimes I_p\Big), \nonumber\\
    &=(\bs{\pi}\mb{1}_n^\top\bs{\pi}\mb{1}_n^\top)\otimes I_p = A_{\infty}. \nonumber
	\end{align}
	Thus~$AA_{\infty}-A_{\infty}A_{\infty}$ is a zero matrix. It can also be verified that~$\frac{1}{n}Y_\infty(\mb{1}_n\otimes I_p)(\mb{1}_n^\top\otimes I_p)=A_\infty$. Based on the discussion above, we have
	\begin{eqnarray}
	A\mb{a}-Y_\infty\overline{\mb{a}}=(A-A_{\infty})(\mb{a}-A_{\infty}\mb{a})=(A-A_{\infty})
	(\mb{a}-Y_\infty\overline{\mb{a}}) \nonumber.
	\end{eqnarray}
	Next we note that~$$\rho(A-A_{\infty})=\rho\Big((\ul{A}-\bs{\pi}\mb{1}_n^\top)\otimes I_p\Big)=\rho(\ul{A}-\bs{\pi}\mb{1}_n^\top)<1,$$ and there exists a matrix norm such that~$\| A-A_{\infty}\| <1$ with a compatible vector norm,~$\|\cdot\|$, see~\cite{hornjohnson:13}: Chapter~5 for details, i.e.,
	\begin{eqnarray}
	\left\|A\mb{a}-Y_{\infty}\overline{\mb{a}}\right\| 
	\leq \|A-A_{\infty}\| \left\|\mb{a}-Y_{\infty}\overline{\mb{a}}\right\|,
	\end{eqnarray}
	and the lemma follows with~$\sigma=\|A-A_{\infty}\|$.
\end{proof}
}
Based on the above notation, we finally denote~$\mb{t}_k$,~$\mb{s}_k\in\mbb{R}^3$, and~$G$,~$H_k\in\mbb{R}^{3\times3}$, for all~$k$ as
{\color{black}
\begin{align}\label{t}
\mb{t}_k&=\left[
\begin{array}{c}
\left\|\mb{x}_k-Y_\infty\overline{\mb{x}}_k\right\| \\
\left\|\overline{\mb{x}}_k-\mb{z}^*\right\|_2 \\
\left\|\mb{w}_k-Y_\infty\mb{g}_k\right\|
\end{array}
\right],
\qquad\mb{s}_k=\left[
\begin{array}{cc}
\left\|\mb{x}_k\right\|_2 \\
0\\
0
\end{array}
\right],
\notag\\
G&=\left[
\begin{array}{ccc}
\sigma & 0 &\alpha \\
\alpha cly_- & \eta & 0\\
cd\epsilon ly_-(\tau+\alpha lyy_-) & \alpha d\epsilon l^2yy_- & \sigma+\alpha cd\epsilon ly_-
\end{array}
\right],
\notag\\
H_k&=\left[
\begin{array}{ccc}
0 & 0 & 0\\
\alpha ly_-T\gamma_1^{k-1} & 0 & 0\\
(\alpha ly+2)d\epsilon ly_-^2T\gamma_1^{k-1} & 0 & 0
\end{array}
\right].
\end{align}
}We now state a key relation of this paper.
\begin{lem}\label{thm1}
	Let the directed graph be strongly-connected and the optimal solution of Problem P1 exist (Assumption \ref{asp2} and \ref{asp3}). Let~$\mb{t}_k$,~$\mb{s}_k$,~$G$, and~$H_k$ be defined in Eq.~\eqref{t}, in which~$\mb{x}_k$ is the sequence generated by ADD-OPT, Eq.~\eqref{alg1_matrix}, over~$k$. Under the smooth and strong-convexity assumptions (Assumption \ref{asp1}), we have~$\mb{t}_k$,~$\mb{s}_k$,~$G$, and~$H_k$ satisfy the following linear relation,
	\begin{align}\label{thm1_eq}
	\mb{t}_k\leq G\mb{t}_{k-1}+H_{k-1}\mb{s}_{k-1}.
	\end{align}
\end{lem}
\begin{proof}
	See Section~\ref{s5}.
\end{proof}
We leave the complete proof to Section \ref{s5}, with the help of several auxiliary relations in Section \ref{s4}. Note that Eq.~\eqref{thm1_eq} provides a linear iterative relation between~$\mb{t}_k$ and~$\mb{t}_{k-1}$ with matrices,~$G$ and~$H_k$. Thus, the convergence of~$\mb{t}_k$ is fully determined by~$G$ and~$H_k$. More specifically, if we want to prove linear convergence of~$\|\mb{t}_k\|_2$ to zero, it is sufficient to show that~$\rho(G)<1$, where~$\rho(\cdot)$ denotes the spectral radius, as well as the linear decaying of~$H_k$, which is straightforward since~$0<\gamma_1<1$. In Lemma \ref{lem_G}, we first show that with appropriate step-size, the spectral radius of~$G$ is less than~$1$. Afterwards, in Lemma~\ref{lem_GH}, we study the convergence properties of the matrices involving~$G$ and~$H_k$.
\begin{lem}\label{lem_G}
Consider the matrix~$G$ defined in Eq.~\eqref{t} as a function of the step-size,~$\alpha$, denoted in this lemma as~$G_\alpha$ to motivate this dependence. It follows that~$\rho(G_\alpha)<1$ if the step-size,~$\alpha\in(0,\overline{\alpha})$, where
	\begin{color}{black}
	\begin{align}\label{alpha_ub}
	\overline{\alpha}=\min\left\{\frac{\sqrt{\Delta^2+4ns(1-\sigma)^2cd\epsilon l^2yy_-^2(l+ns)}-\Delta}{2cd\epsilon l^2yy_-^2(l+ns)},\frac{1}{nl}\right\},
	\end{align}
	and~$\Delta=nscd\epsilon ly_-(1-\sigma+\tau)$,
	where~$c$ and~$d$ are the constants from the equivalence of~$\|\cdot\|$ defined in Lemma~\ref{lem3} and~$\|\cdot\|_2$. 
    \end{color}
\end{lem}
\begin{proof}
	First, if~$\alpha<\frac{1}{nl}$ then~$\eta=1-\alpha ns,$ since~$l\geq s$ (see e.g.,~\cite{opt_literature0}: Chapter 3 for details). When~$\alpha=0$, we have that
	{\color{black}
	\begin{align}
	G_0&=\left[
	\begin{array}{ccc}
	\sigma & 0 & 0 \\
	0 & 1 & 0\\
	cd\epsilon l\tau y_- & 0 & \sigma
	\end{array}
	\right]
	,\end{align}
   }
	the eigenvalues of which are~$\sigma$,~$\sigma$, and~$1$. Hence,~$\rho(G_0)=1$. We now consider how the eigenvalue of~$1$ is changed if we slightly increase~$\alpha$ from~$0$. Let~$\mc{P_{G_\alpha}}(q)=\mbox{det}(qI_n-G_\alpha)$, i.e., the characteristic polynomial of~$G_\alpha$. Setting~$\mbox{det}(qI_n-G_\alpha)=0$, we get the following equation.
	{\color{black}
	\begin{align}\label{det}
	((q-\sigma)^2-\alpha cd\epsilon ly_-(q-\sigma))(q-1+n\alpha s)-\alpha^3cd\epsilon l^3 yy_-^2&\nonumber\\
	-\alpha(q-1+n\alpha s)(cd\epsilon l\tau y_-+\alpha(cd\epsilon l^2yy_-^2))=0&.
	\end{align}
    }Since we have already shown that~$1$ is one of the eigenvalues of~$G_0$, Eq.~\eqref{det} holds when~$q=1$ and~$\alpha=0$. By taking the derivative on both sides of Eq.~\eqref{det}, with~$q=1$ and~$\alpha=0$, we obtain that{\color{black}~$\frac{d q}{d\alpha}|_{\alpha=0,q=1}=-ns<0$.} This leads to the fact that when~$\alpha$ slightly increases from~$0$,~$\rho(G_\alpha)<1$ since the eigenvalues are continuous functions of the parameters of a matrix.
	
	We next calculate all possible values of~$\alpha$ for which~$G_\alpha$ has an eigenvalue of~$1$. Let~$q=1$ in Eq.~\eqref{det} and solve for the step-size,~$\alpha$; we obtain three solutions: ~$\alpha_1=0$,~$\alpha_2<0$, and{\color{black}~$$\alpha_3=\frac{\sqrt{\Delta^2+4ns(1-\sigma)^2cd\epsilon l^2yy_-^2(l+ns)}-\Delta}{2cd\epsilon l^2yy_-^2(l+ns)}>0.$$}Since there are no other values of~$\alpha$ with which~$G_\alpha$ has an eigenvalue of $1$, all eigenvalues of~$G_\alpha$ are less than~$1$, i.e.,~$\rho(G_\alpha)<1$, when~$\alpha\in(0,\overline{\alpha})$.
\end{proof}
\noindent We note that~$\bar{\alpha}$ depends on the global knowledge and it may not be possible to precisely compute it in a distributed fashion. However, this value may be estimated as we will show in Section~\ref{s6}, see e.g.,~\cite{EXTRA}, for a similar approach.

\begin{lem}\label{lem_GH}
	With the step-size,~$\alpha\in(0,\overline{\alpha})$, where~$\overline{\alpha}$ is defined in Eq.~\eqref{alpha_ub}, the following statements hold:~$\forall k$,
	\begin{enumerate}[label=(\alph*)]
		\item
		there exists~$0<\gamma_1<1$ and~$0<\Gamma_1<\infty$, where~$\gamma_1$ is defined in Eq.~\eqref{DkDinfty1}, such that
	    $$\left\|H_k\right\|_2=\Gamma_1\gamma_1^k;$$
		\item 
		there exists~$0<\gamma_2<1$ and~$0<\Gamma_2<\infty$, such that
		$$\left\|G^k\right\|_2\leq\Gamma_2\gamma_2^k;$$
		\item 
		let~$\gamma=\max\{\gamma_1,\gamma_2\}$ and~$\Gamma=\Gamma_1\Gamma_2/\gamma$, such that for all~$0\leq r\leq k-1$,
		$$\left\|G^{k-r-1}H_r\right\|_2\leq\Gamma\gamma^k.$$
	\end{enumerate}
\end{lem}
\begin{proof}~

\begin{enumerate}[label=(\alph*)]
\item This can be verified according to Eq.~\eqref{t} and by letting
    {\color{black}
	$$\Gamma_1=\frac{1}{\gamma_1}\sqrt{(\alpha ly_-T)^2+(\alpha yl+2)^2(d \epsilon ly_-^2T)^2}.$$
    }
{\color{black}
\item Note that $\rho(G)<1$ when $\alpha\in(0,\overline{\alpha})$. Therefore, the value of some matrix norm of~$G$, denoted by $\gamma_2$, is strictly less than 1. Since all matrix norms are equivalent, we have $\|G^k\|_2\leq\Gamma_2\gamma_2^{k}$, for some positive constant $\Gamma_2$.
}

\item The proof of (c) is achieved by combining (a) and (b). 
\end{enumerate}
\end{proof}
{\color{black}
\begin{lem}\label{lem_polyak}
	(Polyak~\cite{polyak1987introduction}) If nonnegative sequences ~$\{v_k\}$,~$\{u_k\}$,~$\{b_k\}$ and~$\{c_k\}$ are such that ~$\sum_{k=0}^{\infty} b_k < \infty$,~$\sum_{k=0}^{\infty} c_k < \infty$ and
	$$v_{k+1} \le (1+b_k)v_k - u_k + c_k, \quad \forall t\ge 0,$$ then~$\{v_k\}$ converges and~$\sum_{k=0}^{\infty} u_k < \infty$.
\end{lem}
}

We now present the main result of this paper in Theorem~\ref{main_result}, which shows the linear convergence rate of ADD-OPT.
\begin{theorem}\label{main_result}
	Let the Assumptions \ref{asp2}-\ref{asp3} hold. With the step-size,~$\alpha\in(0,\overline{\alpha})$, where~$\overline{\alpha}$ is defined in Eq.~\eqref{alpha_ub}, the sequence,~$\{\mb{z}_k\}$, generated by ADD-OPT, converges exactly to the unique optimizer,~$\mb{z}^*$, at a linear rate, i.e., {\color{black}there exist some positive constant~$M>0$, such that for any~$k$,
	\begin{align}
	\left\|\mb{z}_k-\mb{z}^*\right\|_2\leq M(\gamma+\xi)^k,
	\end{align}
	where~$\gamma$ is used in Lemma \ref{lem_GH}(c) and~$\xi$ is a arbitrarily small constant. 
    }
\end{theorem}
\begin{proof}
We write Eq.~\eqref{thm1_eq} recursively, leading to
	\begin{align}\label{thm2_eq1}
	\mb{t}_k\leq&G^k\mb{t}_0+\sum_{r=0}^{k-1}G^{k-r-1}H_r\mb{s}_r.
	\end{align}
By taking the norm on both sides of Eq.~\eqref{thm2_eq1} and considering Lemma \ref{lem_GH}, we obtain that
	\begin{align}\label{thm2_eq2}
	\left\|\mb{t}_k\right\|_2\leq&\left\|G^k\right\|_2\left\|\mb{t}_0\right\|_2+\sum_{r=0}^{k-1}\left\|G^{k-r-1}H_r\right\|_2\left\|\mb{s}_r\right\|_2,\nonumber\\
	\leq&\Gamma_2\gamma_2^k\left\|\mb{t}_0\right\|_2+\sum_{r=0}^{k-1}\Gamma\gamma^k\left\|\mb{s}_r\right\|_2,
	\end{align}
in which we can bound~$\|\mb{s}_r\|_2$ as
\begin{align}
\left\|\mb{s}_r\right\|_2\leq&\left\|\mb{x}_r-Y_\infty\overline{\mb{x}}_r\right\|_2+\left\|Y_\infty\right\|_2\left\|\overline{\mb{x}}_r-\mb{z}^*\right\|_2+\left\|Y_\infty\right\|_2\left\|\mb{z}^*\right\|_2\nonumber,\\
\leq&(c+y)\left\|\mb{t}_r\right\|_2+y\left\|\mb{z}^*\right\|_2.
\end{align}
Therefore, we have that for all~$k$
	\begin{align}\label{thm2_eq3}
	\left\|\mb{t}_k\right\|_2\leq&\bigg{(}\Gamma_2\|\mb{t}_0\|_2+\Gamma(c+y)\sum_{r=0}^{k-1}\|\mb{t}_r\|_2+\Gamma yk\|\mb{z}^*\|_2\bigg{)}\gamma^k.
	\end{align}
%
    {\color{black}
	Denote~$v_k=\sum_{r=0}^{k-1}\|\mathbf{t}_r\|_2$,~$s_k=\Gamma_2\|\mathbf{t}_0\|_2+\Gamma yk\|\mb{z}^*\|_2$, and~$b=\Gamma (c+y)$, then Eq. \eqref{thm2_eq3} can be written as
	\begin{align}\label{thm2_newEq}
	\|\mathbf{t}_{k}\|_2 = v_{k+1} - v_k\leq(s_k + bv_k)\gamma^k,
	\end{align}
	which implies that~$v_{k+1} \le (1+b\gamma^k)v_k + s_k\gamma^k$. Applying Lemma~\ref{lem_polyak} with~$b_k = b\gamma^k$ and~$c_k = s_k\gamma^k$ (here~$u_k=0$), we have that~$v_k$ converges\footnote{In order to apply Lemma~\ref{lem_polyak}, we need to show that~$\sum_{k=0}^\infty s_k\gamma^k<\infty$, which follows from the fact that~$\lim_{k\rightarrow\infty}\frac{s_{k+1}\gamma^{k+1}}{s_{k}\gamma^{k}}=\gamma<1.$}.
	and therefore is bounded. By Eq.~\eqref{thm2_newEq},~$\forall \mu \in (\gamma,1)$ we have
	\begin{equation}
	\lim_{k\rightarrow\infty}\frac{\|\mathbf{t}_{k}\|_2}{\mu^k} \leq
	\lim_{k\rightarrow\infty}\frac{(s_k + bv_k)\gamma^k}{\mu^k}=0.
	\end{equation} 
	Therefore,~$\|\mathbf{t}_{k}\|_2=O(\mu^k)$. In other words, there exists some positive constant~$\Phi$ such that for all~$k$, we have: 
	\begin{align}\label{thm2_eq6}
	\left\|\mb{t}_k\right\|_2\leq&\Phi(\gamma+\xi)^{k},
	\end{align}
	where~$\xi$ is a arbitrarily small constant.
	Moreover,~$\left\|\mb{z}_k-\mb{z}^*\right\|_2$ and~$\left\|\mb{t}_k\right\|_2$ satisfy the relation that
	\begin{align}\label{thm2_eq7}
	\left\|\mb{z}_k-\mb{z}^*\right\|_2\leq&\left\|Y_k^{-1}\mb{x}_k-Y_k^{-1}Y_\infty\overline{\mb{x}}_k\right\|_2+\left\|Y_k^{-1}Y_\infty\mb{z}^*-\mb{z}^*\right\|_2\nonumber\\
	&+\left\|Y_k^{-1}Y_\infty\overline{\mb{x}}_k-Y_k^{-1}Y_\infty\mb{z}^*\right\|_2,\nonumber\\
	\leq&y_-(c+y)\left\|\mb{t}_k\right\|_2+y_-T\gamma_1^{k}\left\|\mb{z}^*\right\|_2,
	\end{align}
	where in the second inequality we use the relation~$$\|Y_{k}^{-1}Y_\infty-I_{np}\|_2\leq\|Y_{k}^{-1}\|_2\|Y_\infty-Y_{k}\|_2\leq y_-T\gamma_1^{k},$$ achieved from Eq. \eqref{DkDinfty1}. By combining Eqs.~\eqref{thm2_eq6} and \eqref{thm2_eq7}, we obtain that
	\begin{align}
	\left\|\mb{z}_k-\mb{z}^*\right\|_2\leq&\Big(y_-(c+y)\Phi+y_-T\|\mb{z}^*\|_2\Big)(\gamma+\xi)^k,\nonumber
	\end{align}
	where $\xi$ is a arbitrarily small constant.
	The proof of theorem is completed by letting~$M=y_-(c+y)\Phi+y_-T\|\mb{z}^*\|_2$.
    }
\end{proof}
\noindent Theorem \ref{main_result} shows the linear convergence rate of ADD-OPT. Although ADD-OPT works for a small enough step-size, how small is sufficient may require some estimation of the upper bound, which we discuss this in Section~\ref{s6}. This notion of sufficiently small step-sizes is not uncommon in the literature, see e.g.,~\cite{uc_Nedic, opdirect_Nedic}. Next, each agent must agree on the same value of step-size that may be pre-programmed to avoid implementing an agreement protocol. We now prove Lemma \ref{thm1} in Sections \ref{s4} and \ref{s5}.

\section{Auxiliary Relations}\label{s4}
We provide several basic relations in this section, which will help the proof of Lemma \ref{thm1}. Lemma \ref{w-x-} derives iterative equations that govern the average sequences,~$\overline{\mb{x}}_k$ and~$\overline{\mb{w}}_k$. Lemma \ref{yy-} gives inequalities that are direct consequences of Eq.~\eqref{DkDinfty1}. Lemma~\ref{gd_approach} can be found in the standard optimization literature, see e.g.,~\cite{opt_literature0}. It states that if we perform a gradient-descent step with a fixed step-size for a smooth, strongly-convex function, then the distance to optimizer shrinks by at least a fixed ratio.
	
\begin{lem}\label{w-x-}
	Recall $\overline{\mb{x}}_k$ from Eq.~\eqref{eqxb} and~$\overline{\mb{w}}_k$ from Eq.~\eqref{eqwb}. The following equations hold for all~$k$,
	\begin{enumerate}[label=(\alph*)]
		\item
		$\overline{\mb{w}}_k=\mb{g}_k$;
		\item 
		$\overline{\mb{x}}_{k+1}=\overline{\mb{x}}_k-\alpha\mb{g}_k$.
	\end{enumerate}
\end{lem}
\begin{proof}
	Since~$A$ is column-stochastic, satisfying~$(\mb{1}_n^\top\otimes I_p) A=\mb{1}_n^\top\otimes I_p$, we obtain that
	\begin{align}
	\overline{\mb{w}}_k&=\frac{1}{n}(\mb{1}_n\otimes I_p)(\mb{1}_n^\top\otimes I_p)\left(A\mb{w}_{k-1}+\nabla\mb{f}_{k}-\nabla\mb{f}_{k-1}\right),\nonumber\\
	&=\overline{\mb{w}}_{k-1}+\mb{g}_k-\mb{g}_{k-1}.\nonumber
	\end{align}
	Do this recursively, and we have that
	\begin{align}
	\overline{\mb{w}}_k=\overline{\mb{w}}_{0}+\mb{g}_k-\mb{g}_{0}.\nonumber
	\end{align}
	Recall that we have the initial condition that~$\mb{w}_0=\nabla\mb{f}_0$, which is equivalent to~$\overline{\mb{w}}_0=\mb{g}_0$. Hence, we achieve the result of (a). The proof of (b) is obtained by the following derivation,
	\begin{align}
	\overline{\mb{x}}_{k+1}&=\frac{1}{n}(\mb{1}_n\otimes I_p)(\mb{1}_n^\top\otimes I_p)\left(A\mb{x}_{k}-\alpha\mb{w}_k\right)\nonumber\\
	&=\overline{\mb{x}}_k-\alpha\overline{\mb{w}}_k,\nonumber\\
	&=\overline{\mb{x}}_k-\alpha\mb{g}_k,\nonumber
	\end{align}
	where the last equation uses the result of (a).
\end{proof}

\begin{lem}\label{yy-}
	Recall Lemma~\ref{lem2},~$Y_k$ from Eq.~\eqref{Yk}, and $Y_\infty$ from Eq.~\eqref{yinf_eq}. The following inequalities hold for all~$k\geq 1$,
	\begin{enumerate}[label=(\alph*)]
		\item
		{\color{black}$\left\|Y_{k-1}^{-1}Y_\infty-I_{np}\right\|_2\leq y_-T\gamma_1^{k-1}$;}
		\item 
		$\left\|Y_{k}^{-1}-Y_{k-1}^{-1}\right\|_2\leq 2y_-^2T\gamma_1^{k-1}$,
	\end{enumerate}
where~$y_-$ is defined in Eq.~\eqref{parameter6}.
\end{lem}
\begin{proof}
	By considering Eq.~\eqref{DkDinfty1}, it follows that
	{\color{black}
	\begin{align}
	\left\|Y_{k-1}^{-1}Y_\infty-I_{np}\right\|_2\leq\left\|Y_{k-1}^{-1}\right\|_2\left\|Y_\infty-Y_{k-1}\right\|_2\leq y_-T\gamma_1^{k-1}.\nonumber
	\end{align}
    }
	The proof of (b) follows by
	\begin{align}
	\left\|Y_{k}^{-1}-Y_{k-1}^{-1}\right\|_2&\leq\left\|Y_{k-1}^{-1}\right\|_2\left\|Y_{k-1}-Y_{k}\right\|_2\left\|Y_{k}^{-1}\right\|_2,\nonumber\\
	&\leq 2y_-^2T\gamma_1^{k-1},\nonumber
	\end{align}
	which completes the proof.
\end{proof}
\begin{lem}\label{gd_approach}
(Bubeck~\cite{opt_literature0})
	Let Assumption \ref{asp1} hold for the objective functions,~$f_i(\mb{z})$, in Problem P1, and let~$s$ and~$l$ be the strong-convexity and Lipschitz-continuity constants, respectively. For any~$\mb{z}\in\mbb{R}^p$, define~$\mb{z}_+=\mb{z}-\alpha\nabla \mb{f}(\mb{z})$, {\color{black}where~$0<\alpha<\frac{2}{nl}$. Then~$$\left\|\mb{z}_+-\ul{\mb{z}}^*\right\|_2\leq\eta\left\|\mb{z}-\ul{\mb{z}}^*\right\|_2,$$ where~$\eta=\max\left(\left|1-\alpha nl\right|,\left|1-\alpha ns\right|\right)$.}
\end{lem}

\section{Convergence Analysis}\label{s5}
We now provide the proof of Lemma \ref{thm1}. We will bound~$\|\mb{x}_k-Y_\infty\overline{\mb{x}}_k\|$,~$\|\overline{\mb{x}}_k-\mb{z}^*\|_2$, and~$\|\mb{w}_k-Y_\infty\mb{g}_k\|$, linearly in terms of their past values, i.e.,~$\|\mb{x}_{k-1}-Y_\infty\overline{\mb{x}}_{k-1}\|$,~$\|\overline{\mb{x}}_{k-1}-\mb{z}^*\|_2$, and~$\|\mb{w}_{k-1}-Y_\infty\mb{g}_{k-1}\|$, as well as~$\|\mb{x}_{k-1}\|_2$. The coefficients are the entries of~$G$ and~$H_{k-1}$.

\textbf{Step 1:} Bound~$\|\mb{x}_k-Y_\infty\overline{\mb{x}}_k\|$. \\
According to Eq.~\eqref{alg1_ma} and Lemma \ref{w-x-}(b), we obtain that
\begin{align}
\left\|\mb{x}_k-Y_\infty\overline{\mb{x}}_k\right\|\leq&\left\|A\mb{x}_{k-1}-Y_\infty\overline{\mb{x}}_{k-1}\right\|\nonumber\\
&+\alpha\left\|\mb{w}_{k-1}-Y_\infty\mb{g}_{k-1}\right\|.
\end{align}
Noticing that~$\|A\mb{x}_{k-1}-Y_\infty\overline{\mb{x}}_{k-1}\|\leq\sigma\|\mb{x}_{k-1}-Y_\infty\overline{\mb{x}}_{k-1}\|$ from Eq. \eqref{sigma_eq}, we have
\begin{align}\label{step1}
\left\|\mb{x}_k-Y_\infty\overline{\mb{x}}_k\right\|\leq&\sigma\left\|\mb{x}_{k-1}-Y_\infty\overline{\mb{x}}_{k-1}\right\|\nonumber\\
&+\alpha\left\|\mb{w}_{k-1}-Y_\infty\mb{g}_{k-1}\right\|.
\end{align}

\textbf{Step 2:} Bound~$\|\overline{\mb{x}}_k-\mb{z}^*\|_2$. \\
By considering Lemma \ref{w-x-}(b), we obtain that
\begin{align}
\overline{\mb{x}}_k=\left[\overline{\mb{x}}_{k-1}-\alpha\mb{h}_{k-1}\right]-\alpha\left[\mb{g}_{k-1}-\mb{h}_{k-1}\right].
\end{align}
Let~$\mb{x}_+=\overline{\mb{x}}_{k-1}-\alpha\mb{h}_{k-1}$, which is a (centralized) gradient-descent step with respect to the global objective function in Problem P1. Therefore, from Lemma~\ref{gd_approach},
\begin{align}
\left\|\mb{x}_+-\mb{z}^*\right\|_2\leq\eta\left\|\overline{\mb{x}}_{k-1}-\mb{z}^*\right\|_2.
\end{align}
From the Lipschitz-continuity, Assumption \ref{asp1}(a), we obtain
{\color{black}
\begin{align}
\left\|\mb{g}_{k-1}-\mb{h}_{k-1}\right\|_2&\leq\left\|\frac{1}{n}(\mb{1}_n\mb{1}_n^\top)\otimes I_p
\right\|_2l\left\|\mb{z}_{k-1}-\overline{\mb{x}}_{k-1}\right\|_2.
\end{align}
}
Therefore, it follows that
\begin{align}\label{s2_1}
\left\|\overline{\mb{x}}_k-\mb{z}^*\right\|_2&\leq\left\|\mb{x}_+-\mb{z}^*\right\|_2+\alpha\left\|\mb{g}_{k-1}-\mb{h}_{k-1}\right\|_2,\nonumber\\
&\leq\eta\left\|\overline{\mb{x}}_{k-1}-\mb{z}^*\right\|_2+\alpha l\left\|\mb{z}_{k-1}-\overline{\mb{x}}_{k-1}\right\|_2.
\end{align}
From Eq.~\eqref{alg1_mc} and Lemma \ref{yy-}(a), it follows that
{\color{black}
\begin{align}\label{s2_2}
\left\|\mb{z}_{k-1}-\overline{\mb{x}}_{k-1}\right\|_2
\leq&\left\|Y_{k-1}^{-1}\left(\mb{x}_{k-1}-Y_\infty\overline{\mb{x}}_{k-1}\right)\right\|_2\nonumber\\
&+\left\|\left(Y_{k-1}^{-1}Y_\infty-I_{np}\right)\overline{\mb{x}}_{k-1}\right\|_2,\nonumber\\
\leq&y_-\left\|\mb{x}_{k-1}-Y_\infty\overline{\mb{x}}_{k-1}\right\|_2\nonumber\\
&+y_-T\gamma_1^{k-1}\left\|\mb{x}_{k-1}\right\|_2,
\end{align}
}where in the second inequality we also make use of the relation~$\|\overline{\mb{x}}_{k-1}\|_2\leq\|\mb{x}_{k-1}\|_2$.
By substituting Eq.~\eqref{s2_2} into Eq.~\eqref{s2_1}, we obtain that
\begin{align}\label{step2}
\left\|\overline{\mb{x}}_k-\mb{z}^*\right\|_2\leq&\alpha cly_-\left\|\mb{x}_{k-1}-Y_\infty\overline{\mb{x}}_{k-1}\right\|+\eta\left\|\overline{\mb{x}}_{k-1}-\mb{z}^*\right\|_2\nonumber\\
&+\alpha ly_-T\gamma_1^{k-1}\left\|\mb{x}_{k-1}\right\|_2.
\end{align}

\textbf{Step 3:} Bound~$\|\mb{w}_k-Y_\infty\mb{g}_k\|$. \\
According to Eq.~\eqref{alg1_md}, we have
\begin{align}
\left\|\mb{w}_k-Y_\infty\mb{g}_k\right\|\leq&\left\|A\mb{w}_{k-1}-Y_\infty\mb{g}_{k-1}\right\|\nonumber\\
&+\left\|\left(\nabla \mb{f}_k-\nabla \mb{f}_{k-1}\right)-\left(Y_\infty\mb{g}_k-Y_\infty\mb{g}_{k-1}\right)\right\|.\nonumber
\end{align}
With Lemma \ref{w-x-}(a) and Eq. \eqref{sigma_eq}, we obtain that
\begin{align}
\left\|A\mb{w}_{k-1}-Y_\infty\mb{g}_{k-1}\right\|&=\left\|A\mb{w}_{k-1}-Y_\infty\overline{\mb{w}}_{k-1}\right\|,\nonumber\\
&\leq\sigma\left\|\mb{w}_{k-1}-Y_\infty\overline{\mb{w}}_{k-1}\right\|.
\end{align}
It follows from the definition of~$\mb{g}_k$ that
{\color{black}
\begin{align}
&\left\|\left(\nabla \mb{f}_k-\nabla \mb{f}_{k-1}\right)-\left(Y_\infty\mb{g}_k-Y_\infty\mb{g}_{k-1}\right)\right\|_2\nonumber\\
&=\left\|\left(I_{np}-\frac{1}{n}Y_\infty(\mb{1}_n\otimes I_p)(\mb{1}_n^\top\otimes I_p)\right)\left(\nabla\mb{f}_k-\nabla\mb{f}_{k-1}\right)\right\|_2.
\end{align}
}
Since~$\frac{1}{n}Y_\infty(\mb{1}_n\otimes I_p)(\mb{1}_n^\top\otimes I_p)=A_\infty$, we obtain that
\begin{align}
\left\|\left(\nabla \mb{f}_k-\nabla \mb{f}_{k-1}\right)-\left(Y_\infty\mb{g}_k-Y_\infty\mb{g}_{k-1}\right)\right\|_2\leq&\epsilon l\left\|\mb{z}_k-\mb{z}_{k-1}\right\|_2,\nonumber
\end{align}
where we use the Lipschitz-continuity, Assumption \ref{asp1}(a). Therefore, we have
\begin{align}\label{s3_1}
\left\|\mb{w}_k-Y_\infty\mb{g}_k\right\|\leq&\sigma\left\|\mb{w}_{k-1}-Y_\infty\mb{g}_{k-1}\right\|\nonumber\\
&+d\epsilon l\left\|\mb{z}_k-\mb{z}_{k-1}\right\|_2.
\end{align}
We now bound~$\|\mb{z}_k-\mb{z}_{k-1}\|_2$. Note that
{\color{black}
\begin{align}
\left\|\mb{h}_{k-1}\right\|_2&=\left\|\frac{1}{n}(\mb{1}_n\otimes I_p)(\mb{1}_n^\top\otimes I_p)\nabla\mb{f}(\overline{\mb{x}}_{k-1})\right\|_2 \nonumber\\
&\leq l\left\|\overline{\mb{x}}_{k-1}-\mb{z}^*\right\|_2.
\end{align}
}
As a result, we have
\begin{align}
\left\|Y_k^{-1}\mb{w}_{k-1}\right\|_2\leq&\left\|Y_k^{-1}\left(\mb{w}_{k-1}-Y_\infty\mb{g}_{k-1}\right)\right\|_2
\nonumber\\
&+\left\|Y_k^{-1}Y_\infty\mb{h}_{k-1}\right\|_2\nonumber\\
&+\left\|Y_k^{-1}Y_\infty\left(\mb{g}_{k-1}-\mb{h}_{k-1}\right)\right\|_2,\nonumber\\
\leq&y_-\left\|\mb{w}_{k-1}-Y_\infty\mb{g}_{k-1}\right\|_2 \nonumber\\
&+y_-yl\left\|\overline{\mb{x}}_{k-1}-\mb{z}^*\right\|_2\nonumber\\
&+y_-yl\left\|\mb{z}_{k-1}-\overline{\mb{x}}_{k-1}\right\|_2,\nonumber\\
\leq&y_-\left\|\mb{w}_{k-1}-Y_\infty\mb{g}_{k-1}\right\|_2 \nonumber\\
&+y_-yl\left\|\overline{\mb{x}}_{k-1}-\mb{z}^*\right\|_2\nonumber\\
&+y_-^2yl\left\|\mb{x}_{k-1}-Y_\infty\overline{\mb{x}}_{k-1}\right\|_2\nonumber\\
&+y_-^2ylT\gamma_1^{k-1}\left\|\mb{x}_{k-1}\right\|_2,
\end{align}
where the last inequality holds due to Eq.~\eqref{s2_2}.
With the upper bound of~$\|Y_k^{-1}\mb{w}_{k-1}\|_2$ provided in the preceding relation and {\color{black}the equality that~$(A-I_{np})Y_\infty\overline{\mb{x}}_{k-1}=\mb{0}_n$,} we can bound~$\|\mb{z}_k-\mb{z}_{k-1}\|_2$ as follows.
\begin{align}
\left\|\mb{z}_k-\mb{z}_{k-1}\right\|_2\leq&\left\|Y_k^{-1}\left(\mb{x}_{k}-\mb{x}_{k-1}\right)\right\|_2\nonumber\\
&+\left\|\left(Y_k^{-1}-Y_{k-1}^{-1}\right)\mb{x}_{k-1}\right\|_2,\nonumber\\
\leq&\left\|Y_k^{-1}\left(A-I_{np}\right)\mb{x}_{k-1}\right\|_2+\alpha\left\|Y_k^{-1}\mb{w}_{k-1}\right\|_2\nonumber\\
&+\left\|Y_k^{-1}-Y_{k-1}^{-1}\right\|_2\left\|\mb{x}_{k-1}\right\|_2,\nonumber\\
\leq&(y_-\tau+\alpha y_-^2yl)\left\|\mb{x}_{k-1}-Y_\infty\overline{\mb{x}}_{k-1}\right\|_2\nonumber\\
&+\alpha y_-\left\|\mb{w}_{k-1}-Y_\infty\mb{g}_{k-1}\right\|_2\nonumber\\
&+\alpha y_-yl\left\|\overline{\mb{x}}_{k-1}-\mb{z}^*\right\|_2\nonumber\\
&+(\alpha yl+2)y_-^2T\gamma_1^{k-1}\left\|\mb{x}_{k-1}\right\|_2.\label{s3_2}
\end{align}
By substituting Eq.~\eqref{s3_2} in Eq.~\eqref{s3_1}, we obtain that
\begin{align}\label{step3}
\left\|\mb{w}_k-Y_\infty\mb{g}_k\right\|\leq&(cd\epsilon l\tau y_-+\alpha cd\epsilon l^2yy_-^2)\left\|\mb{x}_{k-1}-Y_\infty\overline{\mb{x}}_{k-1}\right\|\nonumber\\
&+\alpha d\epsilon l^2yy_-\left\|\overline{\mb{x}}_{k-1}-\mb{z}^*\right\|_2\nonumber\\
&+(\sigma+\alpha cd\epsilon ly_-)\left\|\mb{w}_{k-1}-Y_\infty\mb{g}_{k-1}\right\|\nonumber\\
&+(\alpha yl+2)d\epsilon ly_-^2T\gamma_1^{k-1}\left\|\mb{x}_{k-1}\right\|_2.
\end{align}

\textbf{Step 4:} By combining Eqs.~\eqref{step1} in step 1, \eqref{step2} in step 2, and \eqref{step3} in step 3, we complete the proof.
\section{Numerical Experiments}\label{s6}
In this section, we analyze the performance of ADD-OPT. Our numerical experiments are based on the distributed logistic regression problem over a directed graph:
\begin{align}
\mb{z}^*=\underset{\mb{z}\in\mbb{R}^p}{\operatorname{argmin}}\left(\frac{\beta}{2}\|\mb{z}\|_2^2+\sum_{i=1}^n\sum_{j=1}^{m_i}\ln\left[1+\exp\left(-b_{ij}\mb{c}_{ij}^\top\mb{z}\right)\right]\right).\nonumber
\end{align}
Each agent~$i$ has access to~$m_i$ training examples,~$(\mb{c}_{ij},b_{ij})\in\mbb{R}^p\times\{-1,+1\}$, where~$\mb{c}_{ij}$ includes the~$p$ features of the~$j$th training example of agent~$i$ and~$b_{ij}$ is the corresponding label. This problem can be formulated in the form of Problem P1 with the local objective function,~$f_i$, being
\begin{align}
f_i=\frac{\beta}{2n}\|\mb{z}\|_2^2+\sum_{j=1}^{m_i}\ln\left[1+\exp\left(-\left(\mb{c}_{ij}^\top\mb{z}\right)b_{ij}\right)\right].\nonumber
\end{align}
In our setting, we have~$n=10$,~$m_i=10$, for all~$i$, and~$p=3$.
\subsection{Convergence rate}
In our first experiment, we compare the convergence rate of algorithms that solve the above distributed consensus optimization problem over directed graphs, including ADD-OPT, DEXTRA,~\cite{DEXTRA}, Gradient-Push,~\cite{opdirect_Nedic}, Directed-Distributed Gradient Descent,~\cite{D-DGD}, and the Weight Balanced-Distributed Gradient Descent,~\cite{opdirect_Makhdoumi}. The network topology is described in Fig. \ref{fig1}, where we apply the weighting strategy from Eq.~\eqref{aa}.
\begin{figure}[!h]
	\begin{center}
		\noindent
		\includegraphics[width=2in]{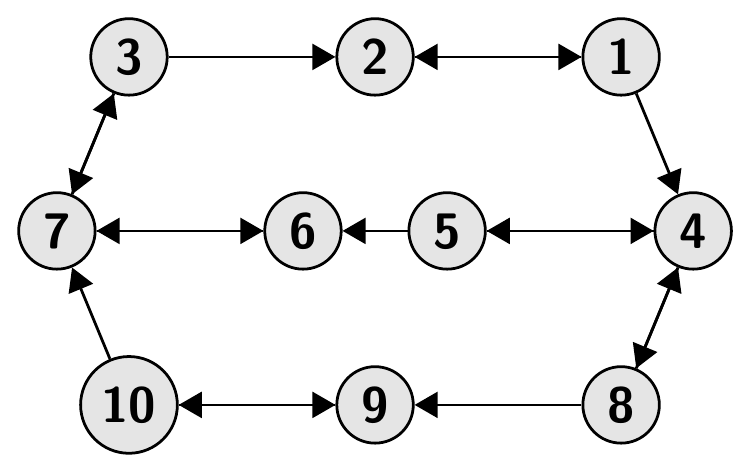}
		\caption{A strongly-connected directed network.}\label{fig1}
	\end{center}
\end{figure} 
The step-size used in Gradient-Push, Directed-Distributed Gradient Descent, and Weight Balanced-Distributed Gradient Descent is~$\alpha_k=1/\sqrt{k}$. The constant step-size used in DEXTRA and ADD-OPT is~$\alpha=0.3$. The convergence rates for these algorithms are shown in Fig.~\ref{fig2}. It shows that ADD-OPT and DEXTRA have a fast linear convergence rate, while other methods are sub-linear. 
\begin{figure}[!h]
	\begin{center}
		\noindent
		\includegraphics[width=3.5in]{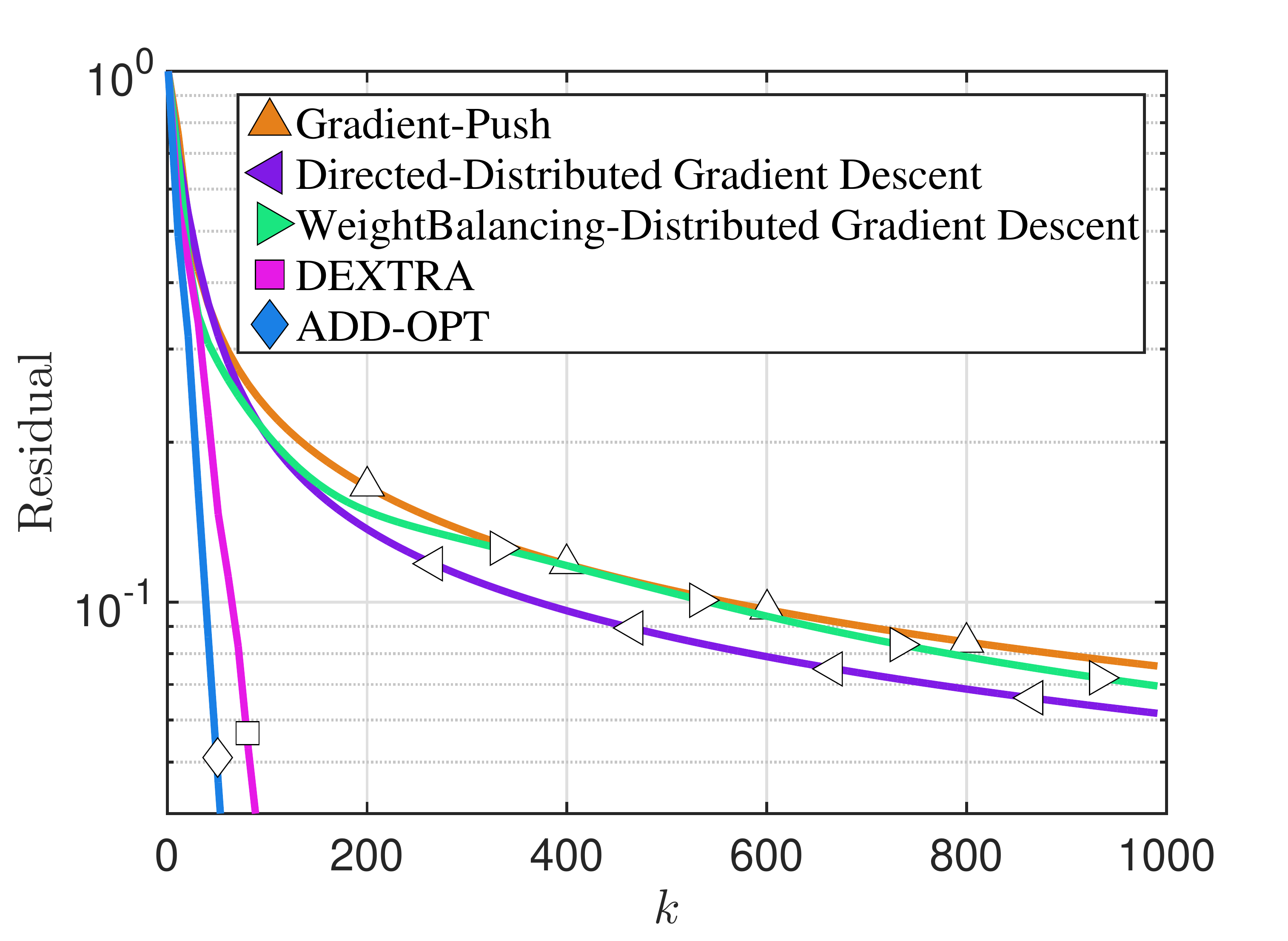}
		\caption{Convergence rates comparison over directed networks.}\label{fig2}
	\end{center}
\end{figure}

\subsection{Step-size range}
We now compare ADD-OPT and DEXTRA in terms of their step-size ranges again with the weighting strategy from Eq.~\eqref{aa}. It is shown in Fig. \ref{fig3} that the greatest lower bound of DEXTRA is around~$\underline{\alpha}=0.2$. In contrast, ADD-OPT works for a sufficiently small step-size. In the given setting, we have~$\tau=1.25$,~$\epsilon=1.11$,~$y=1.96$,~$y_-=2.2$,~$l=1$, and~$\sigma<1$; resulting into~$\overline{\alpha}=\frac{\sqrt{8.7}}{9.57}$, where we choose~$c$ and~$d$ to be~$1$. It can be found in Fig. \ref{fig4} that the practical upper bound of step-size is much bigger, i.e.,~$\overline{\alpha}=1.12$. Since the computation of~$\overline{\alpha}$ is related to the global knowledge, e.g., the network topology, and the strong-convexity and Lipschitz-continuity constants, it is preferable to estimate~$\overline{\alpha}$. According to Eq.~\eqref{alpha_ub}, we have that 
$\overline{\alpha}\approxeq\sqrt{\frac{s(1-\sigma)^2}{\epsilon yy_-^2(l+s)l^2}}$ given that~$\epsilon y(l+s)s(1-\sigma)^2\gg(\epsilon\tau s)^2$.
By estimating~$\tau=\epsilon=y=y_-=1$,~$\sigma=0.9$, and noting that~$s\leq l$, we can estimate~$\overline{\alpha}$ as~$\overline{\alpha}\approxeq\frac{1}{10l}$.
\begin{figure}[!h]
	\begin{center}
		\noindent
		\includegraphics[width=3.5in]{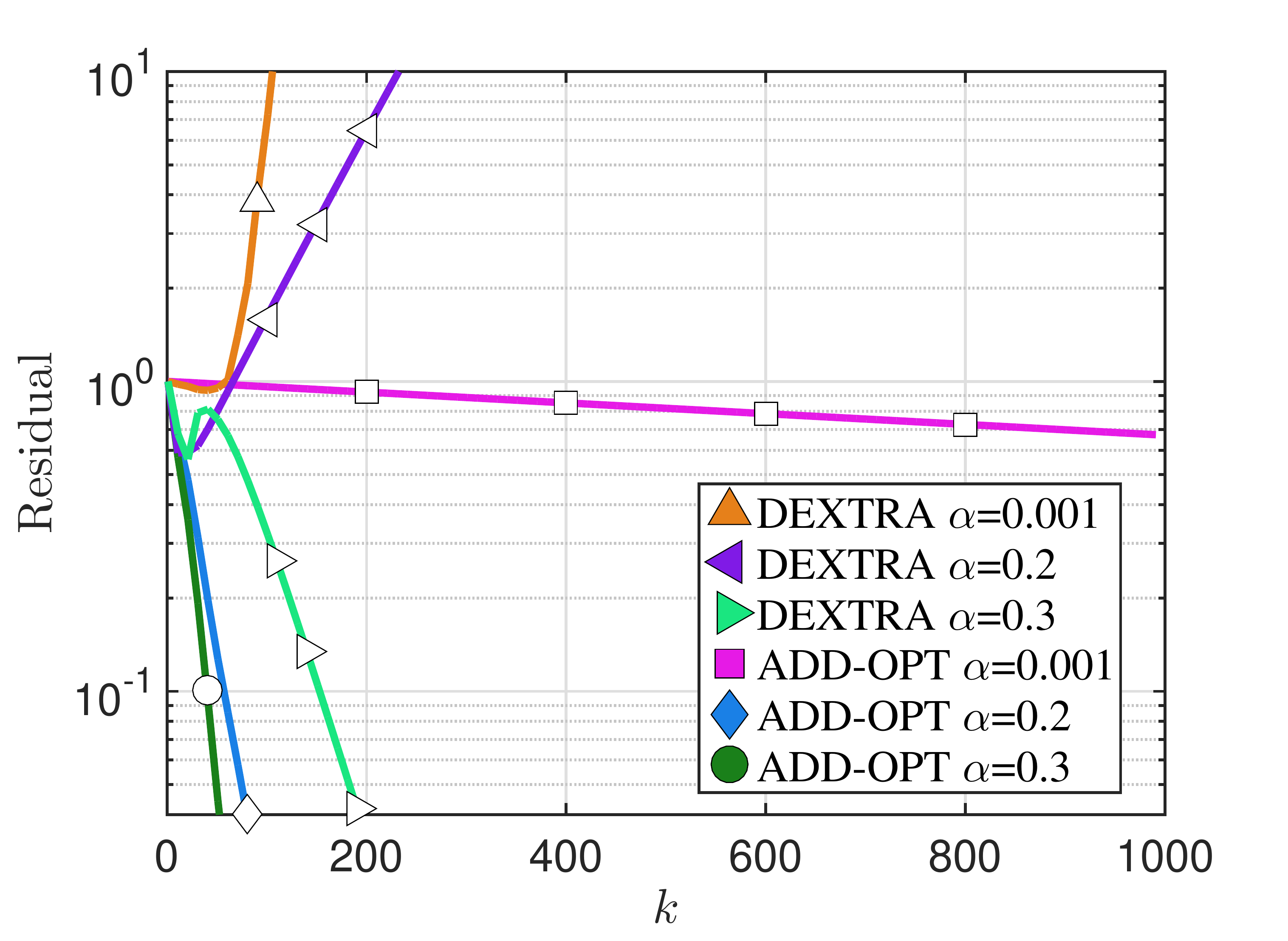}
		\caption{Comparison between ADD-OPT and DEXTRA in terms of step-sizes.}\label{fig3}
	\end{center}
\end{figure}

\subsection{Convergence rate vs. step-sizes}
We note that the convergence rate of ADD-OPT is related to the spectral radius of matrix~$G$, i.e.,~$\rho(G)$, see Eq.~\eqref{thm1_eq}. Therefore, it is possible to achieve the best convergence rate by picking some~$\alpha$ such that the~$\rho(G)$ is minimized. 
\begin{figure}[!h]
	\begin{center}
		\noindent
		\includegraphics[width=3.5in]{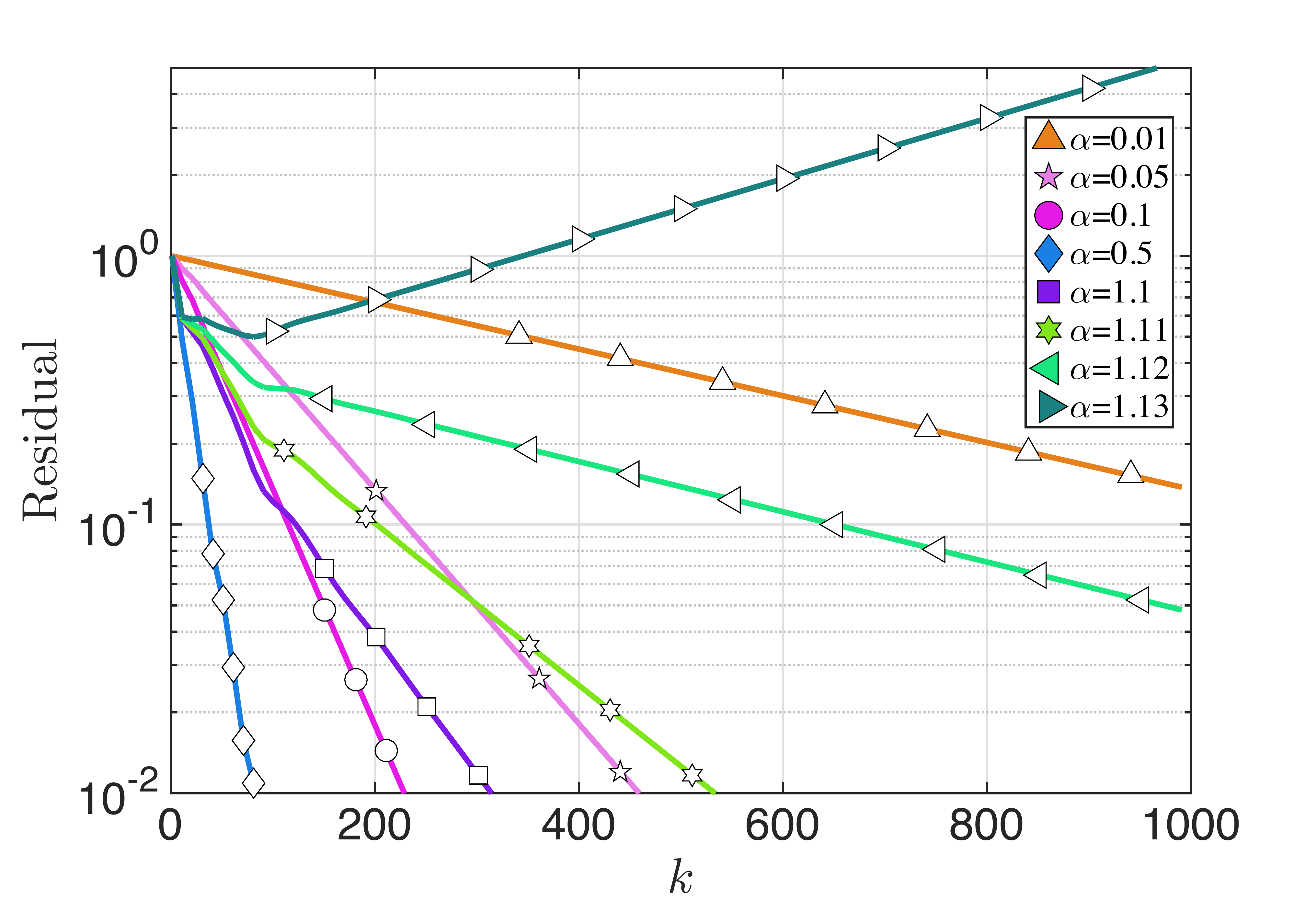}
		\caption{The range of  ADD-OPT 's step-size.}\label{fig4}
	\end{center}
\end{figure}
In Fig. \ref{fig5}, we show the relationship between the spectral radius,~$\rho(G_\alpha)$, of~$G$, and the step-size,~$\alpha$, as well as the residual at the $200$-th iteration,~$\frac{\|\mathbf{z}_{200}-\mathbf{z}^*\|}{\|\mathbf{z}_0-\mathbf{z}^*\|}$, and~$\alpha$. We observe that the best convergence rate is achieved when~$\alpha=0.3$, at which~$\rho(G)$ is minimized. Fig.~\ref{fig5} also demonstrates our previous theoretical analysis in Lemma \ref{lem_G}, where we show that~$\rho(G)=1$, when~$\alpha=0$ or~$\alpha=\overline{\alpha}$, and~$\rho(G)<1$ for~$\alpha\in(0,\overline{\alpha})$. We further note that~$\rho(G_\alpha)<1$, when~$\alpha$ lies approximately in~$(0,0.3)$, which is our theoretical bound of the step-size.
\begin{figure}[!h]
	\begin{center}
		\noindent
		\includegraphics[width=3.5in]{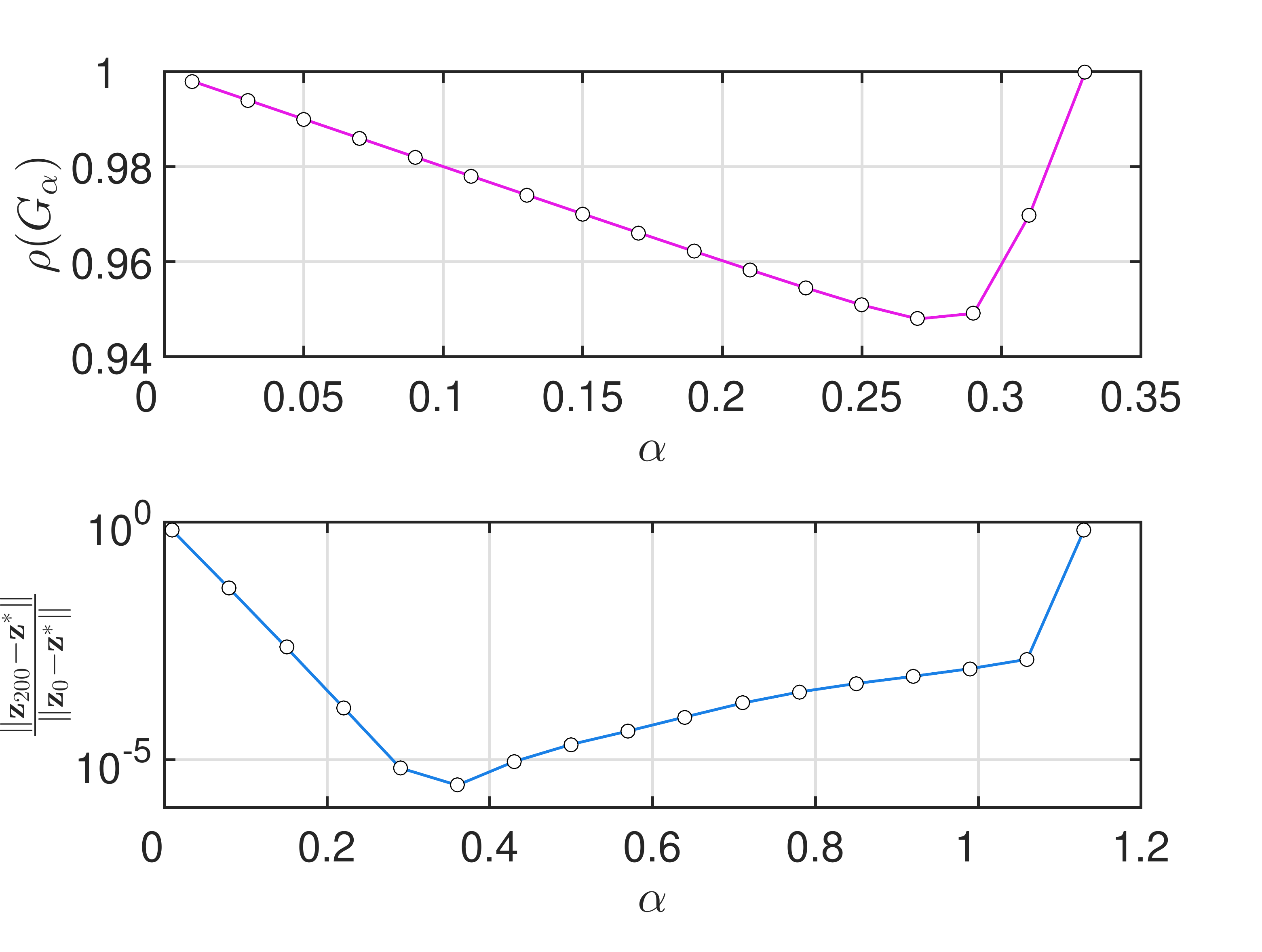}
		\caption{Spectral radius,~$\rho(G_\alpha)$ and the residual at the~$200$th iteration versus~$\alpha$.}\label{fig5}
	\end{center}
\end{figure}

\subsection{Convergence rate vs graph sparsity}
In our last experiment, we observe how does the convergence rate change as a function of the sparsity of the directed graph. We consider three strongly-connected directed graphs as shown in Fig.~\ref{graph}. It can be observed that the residuals decrease faster as the number of edges increases, from~$\mc{G}_a$ to~$\mc{G}_b$ to~$\mc{G}_c$, see Fig.~\ref{fig6}. This indicates faster convergence when there are more communication channels available for information exchange.
\begin{figure}[!h]
\centering
\subfigure{\includegraphics[width=2in]{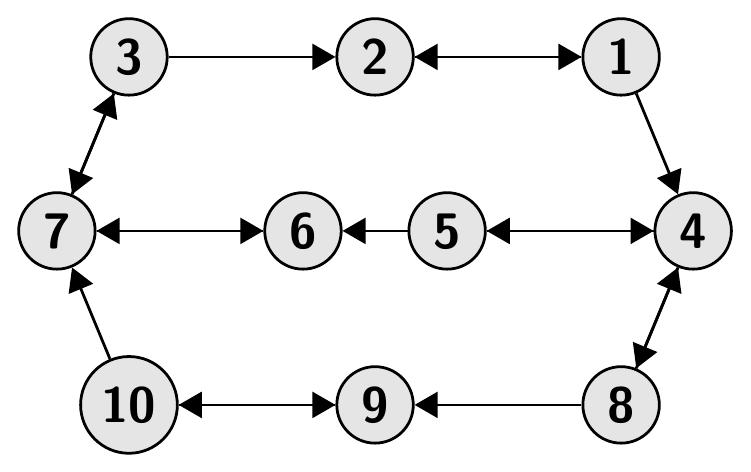}}
\subfigure{\includegraphics[width=2in]{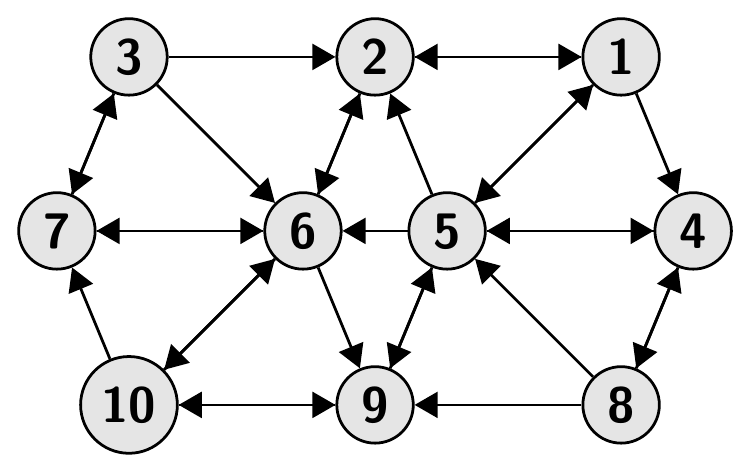}}
\subfigure{\includegraphics[width=2in]{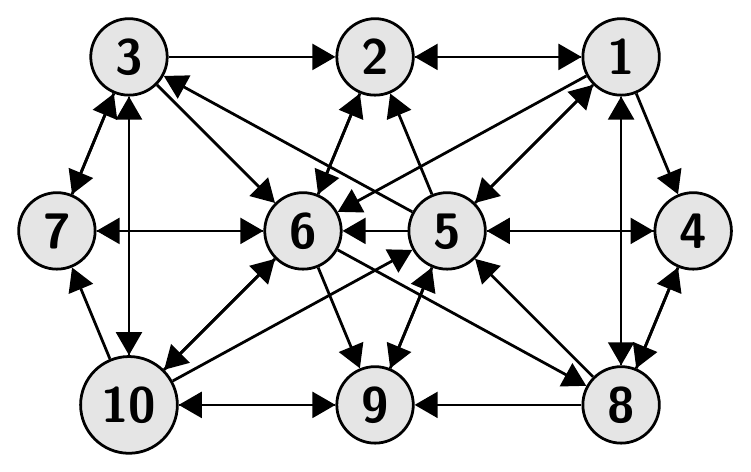}}
\caption{Three examples of strongly-connected directed graphs.}
\label{graph}
\end{figure}

\section{Conclusions}\label{s7}
In this paper, we focus on solving the distributed optimization problem over directed graphs. The proposed algorithm, termed ADD-OPT (Accelerated Distributed Directed Optimization), can be viewed as an improvement of our recent work, DEXTRA. The proposed algorithm, ADD-OPT, achieves the best known rate of convergence for this class of problems,~$O(\mu^{k}),0<\mu<1$, given that the objective functions are strongly-convex with globally Lipschitz-continuous gradients, where~$k$ is the number of iterations. Moreover, ADD-OPT supports a wider and more realistic range of step-sizes in contrast to the existing work. In particular, we show that ADD-OPT converges for arbitrarily small (positive) step-sizes. Simulations further illustrate our results.
\begin{figure}[!h]
	\begin{center}
		\noindent
		\includegraphics[width=3.5in]{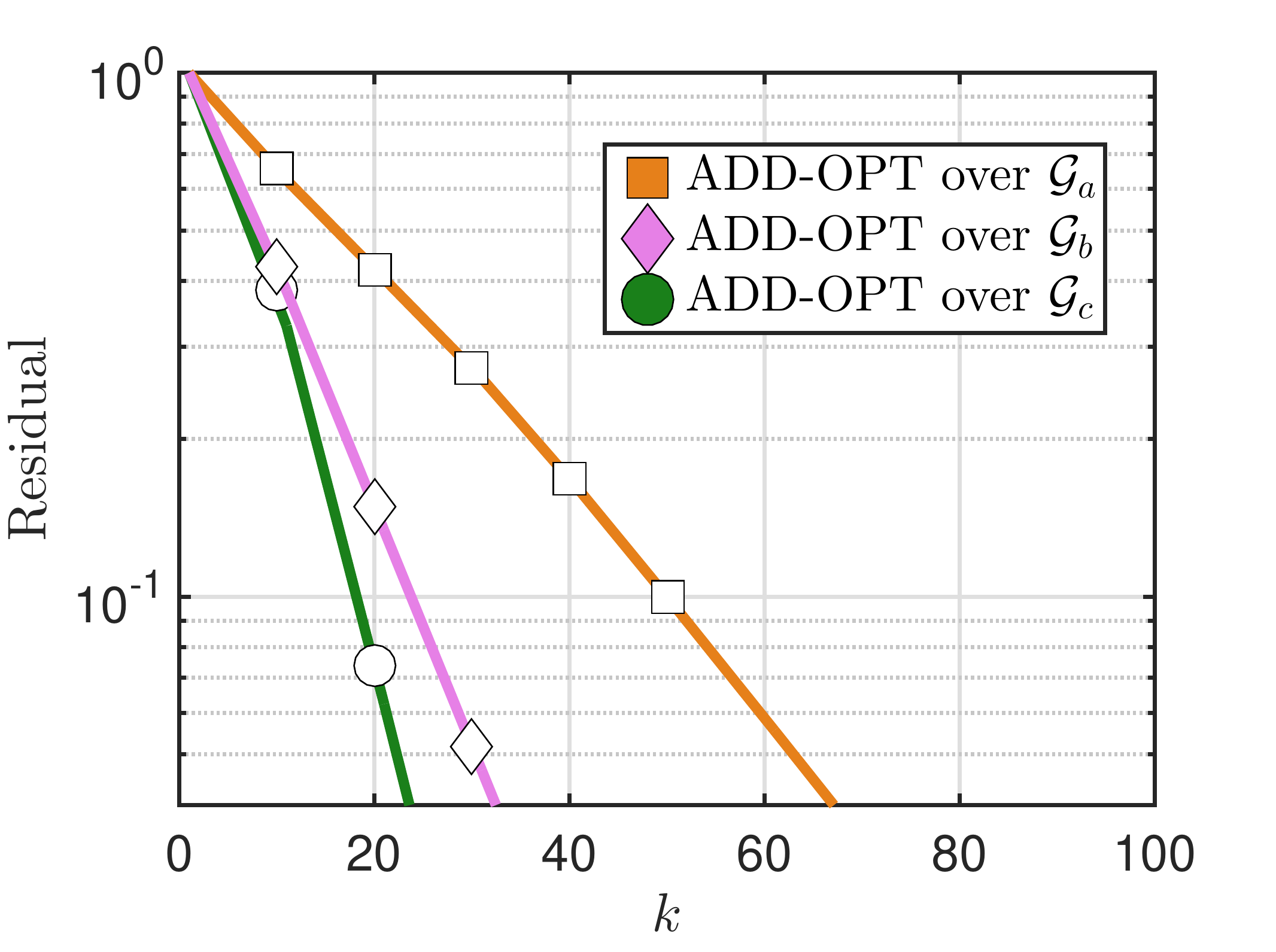}
		\caption{The range of  ADD-OPT 's step-size.}
		\label{fig6}
	\end{center}
\end{figure}

\bibliographystyle{IEEEbib}
\bibliography{sample}

\begin{IEEEbiography}[{\includegraphics[width=1in,height=1.25in,clip,keepaspectratio]{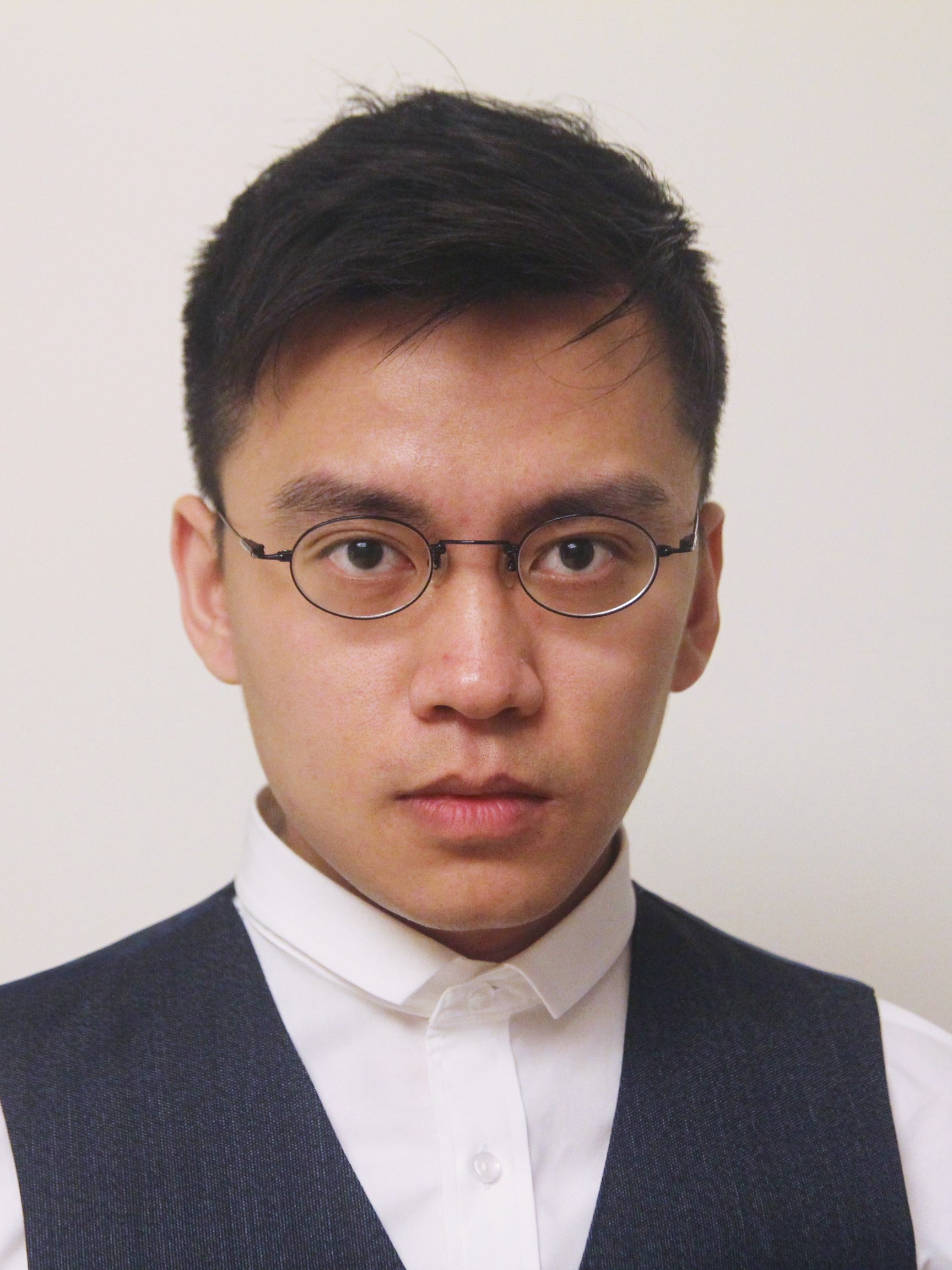}}]{Chenguang Xi} received his B.S. degree in Microelectronics from Shanghai JiaoTong University, China, in 2010, M.S. and Ph.D. degrees in Electrical and Computer Engineering from Tufts University, in 2012 and 2016, respectively. His research interests include distributed optimization, tensor analysis, and source localization.
\end{IEEEbiography}

\begin{IEEEbiography}[{\includegraphics[width=1in,height=1.25in,clip,keepaspectratio]{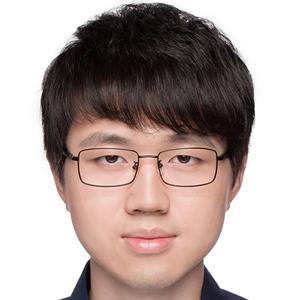}}]{Ran Xin} received his B.S. degree in Mathematics and Applied Mathematics from Xiamen University, China, in 2016. His research interests include distributed optimization and control.
\end{IEEEbiography}

\begin{IEEEbiography}[{\includegraphics[width=1in,height=1.25in,clip,keepaspectratio]{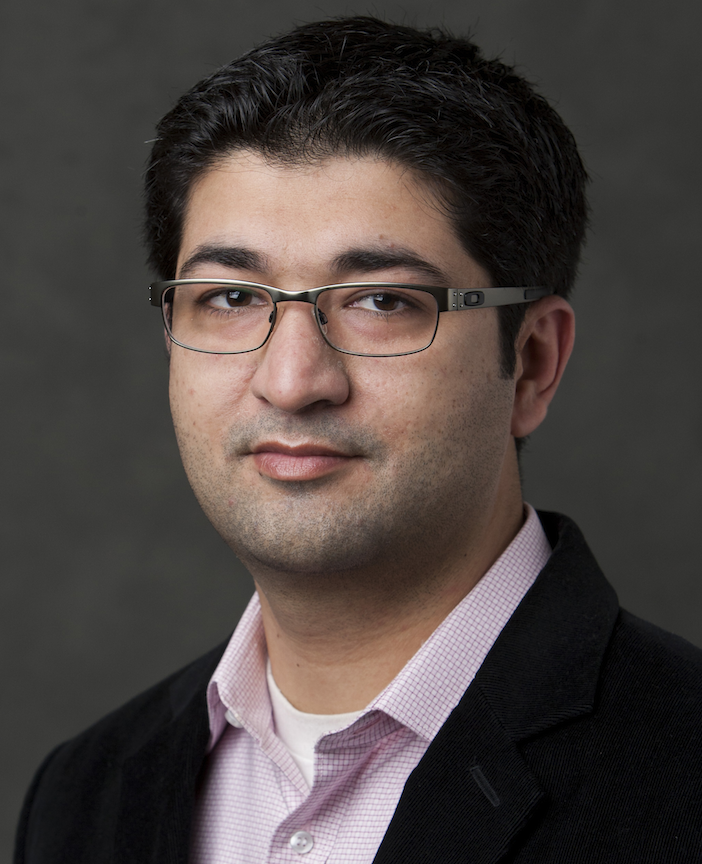}}]{Usman A. Khan} received his B.S. degree (with honors) in EE from University of Engineering and Technology, Lahore-Pakistan, in 2002, M.S. degree in ECE from University of Wisconsin-Madison in 2004, and Ph.D. degree in ECE from Carnegie Mellon University in 2009. Currently, he is an Assistant Professor with the ECE Department at Tufts University. He received the NSF Career award in Jan. 2014 and is an IEEE Senior Member since Feb. 2014. His research interests lie in efficient operation and planning of complex infrastructures and include statistical signal processing, networked control and estimation, and distributed algorithms. Dr. Khan is on the editorial board of IEEE Transactions on Smart Grid and an associate member of Sensor Array and Multichannel Technical Committee with the IEEE Signal Processing Society. He has served on the Technical Program Committees of several IEEE conferences and has organized and chaired several IEEE workshops and sessions. His graduate students have won multiple Best Student Paper awards. His work was presented as Keynote speech at BiOS SPIE Photonics West--Nanoscale Imaging, Sensing, and Actuation for Biomedical Applications IX.
\end{IEEEbiography}

\end{document}